\documentclass[11pt]{aart}

\usepackage{titlesec}
\titleformat{\subsection}[runin]{\normalfont\bfseries}{\thesubsection.}{.5em}{}[.]\titlespacing{\subsection}{0pt}{2ex plus .1ex minus .2ex}{.8em}
\titleformat{\subsubsection}[runin]{\normalfont\itshape}{\thesubsubsection.}{.3em}{}[.]\titlespacing{\subsubsection}{0pt}{1ex plus .1ex minus .2ex}{.5em}
\titleformat{\paragraph}[runin]{\normalfont\itshape}{\theparagraph.}{.3em}{}[.]\titlespacing{\paragraph}{0pt}{1ex plus .1ex minus .2ex}{.5em}

\usepackage[labelfont=sc,font=small,labelsep=period]{caption}
\makeatletter\def\SetFigFont#1#2#3#4#5{\small}

\usepackage[letterpaper, hmargin=1.1in, top=1in, bottom=1.2in, footskip=0.6in]{geometry}

\usepackage{amsmath} 
\usepackage{amssymb}
\usepackage{amsfonts}
\usepackage{latexsym}
\usepackage{amsthm}
\usepackage{amsxtra}
\usepackage{amscd}
\usepackage{mathrsfs}
\usepackage{bm}

\usepackage{graphicx, color}
\usepackage[nottoc,notlof,notlot]{tocbibind} 
\usepackage{cite} 

\numberwithin{equation}{section}
\numberwithin{figure}{section}

\theoremstyle{plain} 
\newtheorem{theorem}{Theorem}[section]
\newtheorem*{theorem*}{Theorem}
\newtheorem{lemma}[theorem]{Lemma}
\newtheorem*{lemma*}{Lemma}
\newtheorem{corollary}[theorem]{Corollary}
\newtheorem*{corollary*}{Corollary}
\newtheorem{proposition}[theorem]{Proposition}
\newtheorem*{proposition*}{Proposition}
\newtheorem{definition}[theorem]{Definition}
\newtheorem*{definition*}{Definition}

\newtheorem*{conjecture*}{Conjecture}

\theoremstyle{definition} 

\newtheorem*{assumption*}{Assumption}

\newtheorem*{example*}{Example}
\newtheorem{remark}[theorem]{Remark}
\newtheorem*{remark*}{Remark}

\newcommand{\f}[1]{\boldsymbol{\mathrm{#1}}} 
\renewcommand{\cal}{\mathcal}

\newcommand{\txt}[1]{\text{\rm{#1}}}

\definecolor{darkred}{rgb}{0.9,0,0.3}
\definecolor{darkblue}{rgb}{0,0.3,0.9}
\definecolor{darkgreen}{rgb}{0,0.7,0.2}

\usepackage{ifthen}
\def\comment#1{\ifthenelse{\isodd{\value{page}}}{\marginpar{\raggedright\scriptsize{\textcolor{darkred}{#1}}}}{\marginpar{\raggedleft\scriptsize{\textcolor{darkred}{#1}}}}}

\renewcommand{\P}{\mathbb{P}}
\newcommand{\E}{\mathbb{E}}
\newcommand{\R}{\mathbb{R}}
\newcommand{\C}{\mathbb{C}}
\newcommand{\N}{\mathbb{N}}
\newcommand{\Z}{\mathbb{Z}}

\def\clap#1{\hbox to 0pt{\hss#1\hss}}

\newcommand{\ee}{\mathrm{e}}
\newcommand{\ii}{\mathrm{i}}
\newcommand{\dd}{\mathrm{d}}
\newcommand{\col}{\mathrel{\vcenter{\baselineskip0.75ex \lineskiplimit0pt \hbox{.}\hbox{.}}}}
\newcommand*{\deq}{\mathrel{\vcenter{\baselineskip0.65ex \lineskiplimit0pt \hbox{.}\hbox{.}}}=}

\newcommand{\eqdist}{\overset{\text{d}}{=}}

\renewcommand{\leq}{\leqslant}
\renewcommand{\geq}{\geqslant}
\renewcommand{\epsilon}{\varepsilon}

\newcommand{\floor}[1] {\lfloor {#1} \rfloor}

\newcommand{\qq}[1]{[\![{#1}]\!]}

\newcommand{\ind}[1]{\f 1 (#1)}

\newcommand{\p}[1]{({#1})}
\newcommand{\pb}[1]{\bigl({#1}\bigr)}
\newcommand{\pB}[1]{\Bigl({#1}\Bigr)}
\newcommand{\pbb}[1]{\biggl({#1}\biggr)}
\newcommand{\pBB}[1]{\Biggl({#1}\Biggr)}
\newcommand{\pa}[1]{\left({#1}\right)}

\newcommand{\q}[1]{[{#1}]}
\newcommand{\qb}[1]{\bigl[{#1}\bigr]}

\newcommand{\qbb}[1]{\biggl[{#1}\biggr]}
\newcommand{\qBB}[1]{\Biggl[{#1}\Biggr]}

\newcommand{\h}[1]{\{{#1}\}}
\newcommand{\hb}[1]{\bigl\{{#1}\bigr\}}

\newcommand{\abs}[1]{\lvert #1 \rvert}
\newcommand{\absb}[1]{\bigl\lvert #1 \bigr\rvert}
\newcommand{\absB}[1]{\Bigl\lvert #1 \Bigr\rvert}

\newcommand{\absa}[1]{\left\lvert #1 \right\rvert}

\newcommand{\norm}[1]{\lVert #1 \rVert}

\newcommand{\normbb}[1]{\biggl\lVert #1 \biggr\rVert}

\newcommand{\scalar}[2]{\langle{#1} \mspace{2mu}, {#2}\rangle}

\DeclareMathOperator{\tr}{Tr}

\DeclareMathOperator{\im}{Im}

\DeclareMathOperator{\spn}{span}

\begin{document}
\title{Bulk eigenvalue statistics for random regular graphs}
\author{Roland Bauerschmidt\footnote{Harvard University, Department of Mathematics. E-mail: {\tt brt@math.harvard.edu}.} \and
Jiaoyang Huang\footnote{Harvard University, Department of Mathematics. E-mail: {\tt jiaoyang@math.harvard.edu}.} \and
Antti Knowles\footnote{ETH Z\"urich, Departement Mathematik. E-mail: {\tt knowles@math.ethz.ch}.} \and
Horng-Tzer Yau\footnote{Harvard University, Department of Mathematics. E-mail: {\tt htyau@math.harvard.edu}.}}
\date{August 24, 2016}
\maketitle

\begin{abstract}
We consider the uniform random $d$-regular graph on $N$ vertices,
with $d \in [N^\alpha, N^{2/3-\alpha}]$ for arbitrary $\alpha > 0$.
We prove that in the bulk of the spectrum the local eigenvalue correlation functions
and the distribution of the gaps between consecutive eigenvalues coincide with those of the Gaussian Orthogonal Ensemble.
\end{abstract}


\section{Introduction and results}
\label{sec:intro}

\subsection{Introduction}

The universality of local eigenvalue statistics is one of the central questions in random matrix theory.
Random matrix statistics are believed to apply to very general complex systems, including the zeros of
the Riemann $\zeta$-function on the critical line. However, proofs of random matrix statistics have so far been
limited mostly to matrix ensembles, with the notable exception \cite{MR1659828}.
There are two classes of matrix ensembles for which random matrix statistics
have been established under very general conditions: invariant ensembles and ensembles with independent entries.
For ensembles of random matrices that are invariant under the unitary or orthogonal group (invariant ensembles),
much has been understood via the method of orthogonal polynomials (see e.g.\ \cite{MR1702716,MR1435193,MR1715324}) and, more recently,
general results have been obtained by direct comparision of ensembles with different potentials
\cite{MR3192527,MR3390602,MR3351052}.
For Wigner matrices and generalized Wigner matrices, whose entries are independent and typically nonzero,
the universality problem has also essentially been solved completely
\cite{MR1810949,MR2810797,MR2639734,MR2662426,MR2981427,MR2784665,
MR3372074,CPA:CPA21624,MR2964770}.
For random sparse matrices with independent entries,
significant progress has been made as well.
In particular, for Erd\H{o}s-R\'enyi graphs, in which each edge is chosen independently with
probability $p$,
random matrix statistics for both the bulk eigenvalues and the second
largest eigenvalue of the adjacency matrix
were established in \cite{MR2964770} under the condition $pN \geq N^{2/3+\alpha}$ with any $\alpha>0$.
For the bulk eigenvalues,
the  lower bound on $pN$ was recently  extended  to $pN \geq N^{\alpha}$ for any $\alpha>0$
in \cite{MR3429490},
and GOE statistics for the eigenvalue gaps was also established.
Finally, aside from the approaches discussed above, supersymmetry
has been used to obtain results on the local eigenvalue statistics
for some special classes of distributions with independent entries (see e.g.\ \cite{MR3192170}).
In addition, local random matrix statistics have been established by an analysis of transfer matrices (see e.g.\ \cite{MR3192614}).

In this paper we study random regular graphs, which are not invariant and do not have independent entries.
We show that the eigenvalues of their adjacency matrices obey random matrix statistics in the bulk of the spectrum.
The universality of local eigenvalue statistics for non-invariant matrix ensembles with correlated entries
has recently been studied in a few other cases.
In particular, after the appearance of this paper,
GOE eigenvalue statistics were proved for the Laplacian matrix of sparse Erd\H{o}s-R\'enyi graphs  in \cite{1510.06390},
and for random matrices with certain short-range correlations in \cite{1604.05709,1604.08188};
these results do not cover the hard constraints of random regular graphs.

\subsection{Main results}

Let $A$ be the adjacency matrix of the uniform random $d$-regular graph (RRG) on $N$ vertices,
i.e.\ a uniformly chosen symmetric matrix with entries in $\{0,1\}$ such that
all rows and columns have sum equal to $d$ and all diagonal entries vanish.
For $d \to \infty$ as $N\to\infty$, it is known \cite{MR2999215,1503.08702,MR3025715} that the eigenvalue density
of $(d-1)^{-1/2}A$ converges to the Wigner \emph{semicircle law} whose density is $\varrho(x) \deq \frac{1}{2\pi}\sqrt{[4-x^2]_+}$.
For $d$ at least $(\log N)^4$, three of the authors recently proved a \emph{local semicircle law} for random regular graphs \cite{1503.08702},
giving precise estimates on the Green's function and the eigenvalue density,
down to spectral scales comparable with the typical eigenvalue spacing (up to a logarithmic correction).
In this paper, we consider the
local eigenvalue statistics of random regular graphs in the bulk of the spectrum.

As the adjacency matrix of a $d$-regular graph, the matrix $A$ has the
trivial uniform eigenvector $\f e \deq N^{-1/2}(1, \dots, 1)^*$ with
eigenvalue $d$.  We denote by $\lambda_1 \geq \dots \geq
\lambda_{N-1}$ the ordered nontrivial eigenvalues of $(d-1)^{-1/2}A$,
and by $\E_{\text{RRG}}$ the expectation with respect to the induced
law on $\lambda_1\geq \dots\geq \lambda_{N-1}$. By comparison, we
denote by $\E_{\text{GOE}}$ the expectation with respect to the law of
the ordered eigenvalues $\lambda_1 \geq \dots \geq \lambda_{N-1}$ of
the Gaussian Orthogonal Ensemble (GOE) on $\R^{(N - 1) \times (N - 1)}$,
normalized so that the off-diagonal entries have variance $N^{-1}$.

The typical locations $\gamma_i$ of the eigenvalues under the semicircle law are defined by
\begin{equation} \label{e:gammadef}
  \frac{i}{N} \;=\; \int_{\gamma_i}^2 \varrho(x) \, \dd x \,.
\end{equation}

\begin{theorem} \label{thm:gap}
Fix $\alpha>0$, and suppose that $d \in [N^\alpha, N^{2/3-\alpha}]$. Then, in the limit $N \to \infty$, the bulk gap statistics of the random $d$-regular graph coincide with those of the GOE. More precisely, for any fixed $\kappa > 0$, $n \in \N$, and $\phi \in C_c^\infty(\R^n)$, we have
\begin{equation} \label{e:gap}
  \pb{\E_{\txt{RRG}} - \E_{\txt{GOE}}} \, \phi\pb{N\varrho(\gamma_i)(\lambda_i-\lambda_{i+1}), \dots, N\varrho(\gamma_i)(\lambda_i-\lambda_{i+n})} \;=\; o(1)
\end{equation}
as $N \to \infty$, uniformly in $i \in \qq{\kappa N, (1-\kappa)N}$.
\end{theorem}

Next, let $p_{\#} \equiv p_{\#,N}$ denote the symmetrized joint law of the eigenvalues of the ensemble $\# = \txt{RRG},  \txt{GOE}$.
The correlation functions are defined for $n \in \qq{1, N - 1}$ by
\begin{equation} \label{e:def_corr_func}
  p^{(n)}_\#(\dd\lambda_1, \dots, \dd\lambda_n) \;\deq\; p_\#\pb{\dd\lambda_1,\dots,\dd\lambda_n, \R^{N-1-n}} \,.
\end{equation}

\begin{theorem} \label{thm:corr}
Fix $\alpha>0$, and suppose that $d \in [N^\alpha, N^{2/3-\alpha}]$.
Then, in the limit $N \to \infty$, the locally averaged local correlation functions of the
random $d$-regular graph coincide with those of the GOE.
More precisely, fix a small enough constant $c > 0$, and define $b \equiv b_N \deq N^{-1 + c}$.
Then for any fixed $n \in \N$, $\phi \in C_c^\infty(\R^n)$, and $E \in (-2,2)$ we have
\begin{equation} \label{e:corr}
\frac{1}{2b} \int_{E-b}^{E+b} \dd E'  \int_{\R^n} \phi(x_1, \dots, x_n) 
N^n \pb{p_{\txt{RRG}}^{(n)} - p_{\txt{GOE}}^{(n)}} \pbb{ E'+\frac{\dd x_1}{N\varrho(E)}, \dots, E'+\frac{\dd x_n}{N\varrho(E)} }
\;=\; o(1)
\,.
\end{equation}
\end{theorem}

For the GOE, the eigenvalue correlation functions are known explicitly; see e.g.\ \cite{MR2129906}.
Hence, the quantities for the GOE appearing 
on the left-hand sides of \eqref{e:gap} and \eqref{e:corr}
can be computed explicitly.
In fact, the eigenvalue gap distribution has only been computed in the sense of averages over the gap index; 
for the GUE, the computation for a fixed gap was performed in \cite{MR3101841}.

The proofs of Theorems~\ref{thm:gap}--\ref{thm:corr} follow the general three-step strategy
developed in \cite{MR2810797,MR2639734,MR2662426};
see e.g.\ \cite{MR2917064} for a survey.
In our setup, the strategy is formulated precisely in Section~\ref{sec:strategy}.
The general idea is to study the convergence of eigenvalue statistics under 
Dyson Brownian motion (DBM) \cite{MR0148397}.
The three steps consist of
(i) a \emph{local law} providing precise estimates on the eigenvalue density down to the scale
of individual eigenvalues, as well as the complete delocalization of the eigenvectors;
(ii) the universality of the local eigenvalue statistics after the short time $t=N^{-1+\delta}$;
and (iii) effective approximation of the local eigenvalue statistics of the original
matrix ensemble at $t = 0$ by the one evolved up to time $t=N^{-1+\delta}$.

In all previous instances of the three-step strategy outlined above,
the independence of the matrix entries was crucial for steps (i) and (iii).
For the random regular graph, a new approach is required for both of these steps,
the last one of which is the main content of this paper.
The local law for random regular graphs was recently established in \cite{1503.08702}, thus performing step~(i).
As for step~(ii), the convergence of the local eigenvalue statistics under DBM
with deterministic initial data was recently established in \cite{1504.03605},
under the sole assumption that the eigenvalue density be bounded at the scale $N^{-1+\delta}$. 
Therefore the local semicircle law provides sufficient control
on the eigenvalues so that using \cite{1504.03605, 1503.08702} we can perform step~(ii).

Thus, the main difficulty is step~(iii). 
There are several known methods for performing this step,
including Lindeberg's proof of the central limit theorem combined with higher moment matching conditions \cite{MR2784665}, 
or the Green's function comparison theorem \cite{MR2981427}.
For short times, 
a more direct method is to prove the stability of the eigenvalues under 
the DBM  by analysing the dynamics of the individual matrix entries \cite{BY2016}.
In all of these approaches, the independence of the matrix entries is used in  an essential way. 
In contrast, the entries of random regular graphs are subject to hard constraints,
and are therefore not independent. 
Tracking carefully the dependence of the matrix entries (using the methods from \cite{1503.08702}),
we  find  that the eigenvalue  evolution  
is stable under a \emph{constrained DBM}, for times $t \leq N^{-1 + \delta}$.
Here, by \emph{stability}, we mean that the changes in the local eigenvalue statistics are negligible.

This stability can also be interpreted as follows: there is a class of
reasonably well-behaved observables, which completely characterize the
local bulk eigenvalue statistics, and whose time evolution under the
constrained DBM can be well approximated by a \emph{switching dynamics}
of random regular graphs.  We note that it has been
proposed that, for random regular graphs, the dynamics provided by DBM
should be replaced with a switching dynamics; see in particular
\cite{1501.04907}.
However, obtaining rigorous results on the local eigenvalue statistics
using only a switching dynamics is difficult, because the induced
eigenvalue process is neither continuous nor autonomous \cite{1503.06417}.
Our strategy crucially relies on the fact that
the eigenvalue process under DBM is continuous and satisfies an
autonomous system of SDEs. 

Theorem \ref{thm:gap} and Theorem \ref{thm:corr} hold also for  sparse random matrices with independent entries;
see \cite{MR3429490}.
We will use parts of that analysis which are applicable here.
The main effort and novelty of this paper is in the control of eigenvalues under
constrained DBM up to time $t=N^{-1+\delta}$ using switchings.

\subsection{Further related results}

Large regular graphs have been proposed as a testing ground for
quantum chaos \cite{MR3204183}.
It is conjectured \cite{PhysRevLett.52.1} that chaotic quantum systems
(i.e.\ quantum systems obtained by quantization of ergodic classical systems)
exhibit random matrix statistics.
Regular graphs are random matrices with a local structure, and as such a step in
the direction of understanding highly structured systems.
It is believed that the eigenvalues of random $d$-regular graphs obey random matrix statistics for any $d \geq 3$.
Indeed, there is numerical evidence that
the local spectral statistics in the bulk of the spectrum are governed by those of the GOE \cite{MR1691538,MR2647344},
and further that the distribution of the appropriately rescaled second largest eigenvalue 
converges to the Tracy-Widom distribution of the GOE \cite{MR2433888}.
Our assumption $d \geq N^{\alpha}$ for arbitrary $\alpha>0$ is purely technical,
since some of the results used in our proof have only been established up to multiplicative errors of order $N^c$ (with arbitrary $c>0$).
We believe that our results can be extended to $d \geq (\log N)^{O(1)}$ with the same method.
Furthermore, our results also extend to the other models of random regular
graphs considered in \cite{1503.08702}, such as the permutation model.
Other results about the eigenvalue and eigenvector distribution of
$d$-regular graphs on mesoscopic and macroscopic scales,
with $d \to \infty$ and with $d$ fixed, are discussed in \cite{1503.08702}.

Our proof relies on \emph{switchings} that leave the random regular graph invariant.
Switchings of random regular graphs were introduced to obtain enumeration
results in \cite{MR790916}; see \cite{MR1725006} for a survey of subsequent developments.
They are also used for simulation of random regular graphs; 
see e.g.\ \cite{MR2334585} and references therein.
Recently, switchings were used to
bound the singularity probability of directed random regular graphs \cite{1411.0243}.
They also played an important role in our recent proof of the
local semicircle law for random regular graphs \cite{1503.08702}.

\bigskip

\paragraph{Notation}

We use $a=O(b)$ to mean that there exists an absolute constant $C>0$ such that $\abs{a} \leq C b$,
and $a \gg b$ to mean that $a \geq C b$ for some sufficiently large absolute constant $C>0$.
We use $c$ for an arbitrarily small positive constant that may change from line to line.
Moreover, we abbreviate $\qq{a,b} \deq [a,b] \cap \Z$.
We use the standard notations $a \wedge b \deq \min \{a,b\}$ and $a \vee b \deq \max \{a,b\}$.
Every quantity that is not explicitly a constant may depend on $N$, which we almost always omit from our notation.
Throughout the paper, we tacitly assume $N\gg 1$. Unless otherwise stated, all sums of indices are over the set $\qq{1,N}$.

\section{Strategy of proof}
\label{sec:strategy}

Our goal is to prove that, in the bulk on the spectrum,
the local eigenvalue statistics of $A/\sqrt{d-1}$ are the same as those of the GOE.
As mentioned in Section~\ref{sec:intro}, in order to show this, we interpolate between the RRG and the GOE
using Dyson Brownian motion, or more precisely its Ornstein-Uhlenbeck version.

\subsection{Constrained Dyson Brownian motion}
\label{sec:cDBM}

The adjacency matrix $A$ of a regular graph is subject to the hard constraints
that its rows and columns have fixed sum
(i.e.\ it has the eigenvector $\f e = N^{-1/2}(1,\dots, 1)^*$).
Therefore, instead of the usual Dyson Brownian motion, we use Dyson Brownian motion
constrained to the subspace of symmetric matrices
whose row and column sums vanish.

We begin with the notion of an Ornstein-Uhlenbeck process on a general finite-dimensional space.

\begin{definition} \label{def:OU_process}
Let $\cal H$ be a real finite-dimensional Hilbert space.
Let $(\f f_\alpha)_{\alpha}$ be an orthonormal basis of $\cal H$.
\begin{enumerate}
\item
Let $(w_\alpha)_\alpha$ be i.i.d.\ standard normal random variables.
Then we define the \emph{standard Gaussian measure on $\cal H$} as $W \deq \sum_\alpha w_\alpha \f f_\alpha$.
\item
Let $(h_\alpha)_{\alpha}$ be i.i.d.\ Ornstein-Uhlenbeck processes satisfying
\begin{align*}
  \dd h_{\alpha} \;=\; 
   \dd B_{\alpha}-\frac{1}{2} h_{\alpha} \, \dd t\,,
\end{align*}
where $(B_\alpha)_{\alpha}$ is a family of i.i.d.\ standard Brownian motions.
Then we define the \emph{standard Ornstein-Uhlenbeck process on $\cal H$} as $H(t) \deq \sum_{\alpha} h_\alpha(t) \, \f f_\alpha$.
\end{enumerate}
\end{definition}

It is easy to verify that
the laws of $W$ and the process $H$ do not depend on the choice of the orthonormal basis $(\f f_\alpha)$,
and that the standard Gaussian measure is invariant under the standard Ornstein-Uhlenbeck process.
We use these properties tacitly from now on.

For example, let $\cal H \deq \{ H \in \R^{N\times N} \col H= H^* \}$ be the Hilbert space of real symmetric $N \times N$ matrices
with inner product
\begin{equation} \label{e:inner_prod}
\scalar{X}{Y} \;\deq\;
\frac{N}{2} \tr (X Y)\,.
\end{equation}
Then the usual $N$-dimensional Dyson Brownian motion is the standard Ornstein-Uhlenbeck process $H(t)$ on $\cal H$.
More explicitly, $H(t)$ is the Markov process satisfying the SDE
\begin{equation} \label{e:standard_DBM}
\dd H \;=\; \frac{1}{\sqrt{N}} \, \dd B  - \frac{1}{2} H \, \dd t\,,
\end{equation}
where $B(t)$ is Brownian motion on the space of $N \times N$ real symmetric matrices with
quadratic covariation $\scalar{B_{ij}}{B_{kl}}(t) = (\delta_{ik} \delta_{jl} + \delta_{il} \delta_{jk}) t$.

More intrinsically, given a finite-dimensional Hilbert space $V$,
we denote the Hilbert space of symmetric linear maps on $V$ 
with inner product \eqref{e:inner_prod} by $\cal H(V)$.
Then we define \emph{Dyson Brownian motion (DBM) on $V$} to be the standard
Ornstein-Uhlenbeck process on $\cal H(V)$. With this point of view, the usual $N$-dimensional DBM is the DBM on $V=\R^N$, and the constrained DBM is the DBM on $V = \f e^\perp$.
Note that the normalization $N$ in \eqref{e:inner_prod} does not need to agree with the dimension of $V$,
which is $N-1$ for $V = \f e^\perp$.
We make the convention to always normalize the inner product \eqref{e:inner_prod} by $N$,
no matter the dimension of $V$, and always denote the dimension of $V$ by $M$.
Finally, we denote the inner product on $V$ by $\f v \cdot \f w$ for $\f v, \f w \in V$.

\begin{definition}[Constrained DBM and GOE] \label{def:constr_DBM}
The \emph{constrained DBM} is the DBM on $\f e^\perp$,
i.e.\ the standard Ornstein-Uhlenbeck process on $\cal H(\f e^\perp)$
with inner product \eqref{e:inner_prod}.
The \emph{constrained GOE} is the standard Gaussian measure on $\cal H(\f e^\perp)$ with inner product \eqref{e:inner_prod}.
\end{definition}

Thus, up to a change of basis, the constrained DBM is equivalent to the usual $(N - 1)$-dimensional DBM, with the minor difference of normalization by $N$ rather than $N-1$.
However, since the definition of the $d$-regular graph is tied to the standard basis of $\R^N$,
it is frequently convenient to work with the constrained DBM in the standard basis of $\R^N$.

Next, in accordance with the decomposition $\R^N = \f e^\perp \oplus \spn(\f e)$, we have a canonical isomorphism $H \mapsto \tilde H \deq H \oplus 0$ from $\cal H(\f e^\perp)$ to the set of matrices
\begin{equation}
\cal M \;\deq\; \{ H \in \R^{N\times N} \col H=H^*, H\f e = 0 \}\,.
\end{equation}
Throughout this paper, we tacitly identify $H$ and $\tilde H$.

We denote by $\cal C^n(\cal M)$ the space of functions $F \col \cal M \to \C$
with continuous bounded derivatives up to order $n$.
Sometimes it will be convenient to compute derivatives of functions $F \in \cal C^n(\cal M)$
in directions of $\R^{N \times N}$ that do not lie in $\cal M$, which is made possible by the following convention.

\begin{definition}
Let $P = I-\f e \f e^*$ be the orthogonal projection from $\R^N$ onto $\f e^\perp$.
We extend any function $F \in \cal C^n(\cal M)$ to a $\cal C^n$-function on $\R^{N \times N}$ through
\begin{equation*}
H \;\longmapsto\; F \pbb{\frac12 P(H + H^*)P}\,,
\end{equation*}
and denote this extended function also by $F$. Finally, for any $F \in \cal C^1(\cal M)$ and $i,j \in \qq{1,N}$,
we use the abbreviation $\partial_{ij} F(H) \equiv \frac{\partial F}{\partial H_{ij}} (H)$.
\end{definition}

From now on, we take $W$ to be the constrained GOE and $H \equiv H(t)$ to be the constrained DBM,
with initial condition
\begin{equation} \label{e:Hdef}
  H(0) \;\deq\;
  \frac{1}{\sqrt{d-1}} (A-d \f e \f e^*) \;\in\; \cal M\,.
\end{equation}
Here $A$ is the adjacency matrix of the random $d$-regular graph.
In particular, the eigenvalues of $H(0)$ as an element of $\cal H(\f e^\perp)$ are the rescaled nontrivial eigenvalues of $A$.

\subsection{Switchings}
Simple switchings are an especially convenient generating set of $\cal M$;
they play a central role throughout this paper.
For any $i,j,k,l \in \qq{1,N}$ we define the \emph{simple switching} $\xi_{ij}^{kl} \in \cal M$ by
\begin{equation} \label{e:Xijkl-def}
  \xi_{ij}^{kl} \;\deq\; \Delta_{ij}+\Delta_{kl}-\Delta_{ik}-\Delta_{jl}
  \qquad
  \text{where} \quad
  (\Delta_{ij})_{pq} \;\deq\; \delta_{ip}\delta_{jq} + \delta_{iq}\delta_{jp}
  \,.
\end{equation}
The action of a simple switching $\xi_{ij}^{kl}$ on an adjacency matrix,
given by $A \mapsto A + \xi_{ij}^{kl}$, amounts to adding the edges $\{i,j\}, \{k,l\}$ and removing the edges $\{i,k\}, \{j,l\}$;
this is illustrated in Figure~\ref{fig:switch1} and made precise in \eqref{e:def_switching} below.
In this section, the four vertices need not be distinct.

\begin{figure}[t]
\begin{center}
\input{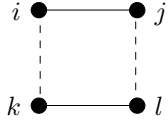}
\end{center}
\caption{A simple switching is given by replacing the solid edges by the dashed edges.
\label{fig:switch1}}
\end{figure}

Next, we define the abbreviations 
\begin{equation}
H_{ij}^{kl} \;\deq\; \tr (\xi_{ij}^{kl} H) \,, \qquad \partial_{ij}^{kl} \;\deq\; \partial_{\xi_{ij}^{kl}} \;=\; \tr (\xi_{ij}^{kl} \partial)\,,
\end{equation}
for all $i,j,k,l \in \qq{1,N}$. Here $\partial_X$ denotes the directional derivative in the direction $X$. Explicitly,
expressed in the standard basis on $\R^N$, we have
\begin{align} \label{e:HXdef}
  H_{ij}^{kl} &\;=\; 2(H_{ij}+H_{kl}-H_{ik}-H_{jl})\,, 
  \\
  \partial_{ij}^{kl} F(H) 
  &\;=\; 2(\partial_{ij}+\partial_{kl}-\partial_{ik}-\partial_{jl}) F(H)\,,
\end{align}
where $F \in \cal C^1(\cal M)$.
With these abbreviations, the generator of the constrained DBM can be expressed in terms of switchings
as stated in the following proposition.

\begin{proposition} \label{prop:cDBM}
The generator of the constrained DBM from Definition \ref{def:constr_DBM} is
\begin{equation} \label{e:L-cDBM}
L \;\deq\; \frac{1}{16N^3} \sum_{i,j,k,l} (\partial_{ij}^{kl})^2 - \frac{1}{32 N^2} \sum_{i,j,k,l} H_{ij}^{kl} \partial_{ij}^{kl}\,.
\end{equation}
This means that for any $F \in \cal C^2(\cal M)$ we have
\begin{equation} \label{e:L_is_generator}
\frac{\dd}{\dd t} \E \q{F(H(t))} \;=\; \E \q{L F(H(t))}\,.
\end{equation}
\end{proposition}

\begin{proof}
Let $\hat H(t)$ be the standard Ornstein-Uhlenbeck process from Definition~\ref{def:OU_process}
on the space $\cal H(\R^{N - 1})$ with inner product \eqref{e:inner_prod}.
As in the example \eqref{e:standard_DBM}, we obtain the quadratic covariation
\begin{equation} \label{e:hatH-qc}
\scalar{\hat H_{ij}}{\hat H_{kl}}(t) \;=\; \frac{1}{N} (\delta_{ik} \delta_{jl} + \delta_{il} \delta_{jk}) t \,.
\end{equation}

Next, let $R \in O(N)$ satisfy $R \f e_N = \f e$. Then, since the inner product \eqref{e:inner_prod}
is invariant under orthogonal conjugations, we can express the constrained DBM as $H(t) = R(\hat H(t) \oplus 0) R^*$.
We abbreviate $H \equiv H(t)$ and write for $F \in \cal C^2(\cal M)$, using It\^{o} calculus,
\begin{equation*}
  \dd \, \E F(H) \;=\; - \frac{1}{2} \sum_{i,j} \E \qb{H_{ij} (\partial_{ij}F)(H)} \,\dd t
  + \frac{1}{2} \sum_{i,j,k,l} \E \qb{(\partial_{ij}\partial_{kl}F)(H) \, \dd \langle H_{ij}, H_{kl}\rangle }\,.
\end{equation*}
By definition of $R$ we have $R_{iN} = \frac{1}{\sqrt{N}}$ for all $i$, so that \eqref{e:hatH-qc} yields 
\begin{align*}
  \dd \langle H_{ij}, H_{kl}\rangle
  &\;=\; \frac{1}{N} \sum_{m,n = 1}^{N-1} \pb{R_{im} R_{jn} R_{km}R_{ln} + R_{im} R_{jn} R_{kn}R_{lm}} \, \dd t
  \nonumber
  \\
  &\;=\;
 \frac{1}{N} \pbb{\delta_{ik}-\frac{1}{N}}\pbb{\delta_{jl}-\frac{1}{N}} \, \dd t +
 \frac{1}{N} \pbb{\delta_{il}-\frac{1}{N}}\pbb{\delta_{jk}-\frac{1}{N}} \, \dd t
  \,.
\end{align*}
Thus, for any $F \in \cal C^2(\cal M)$ we have \eqref{e:L_is_generator} with
\begin{equation} \label{e:gen1}
L \;=\; \frac{1}{N^3} \sum_{i,j,k,l} \partial_{ij}(\partial_{ij} + \partial_{kl} - \partial_{il} - \partial_{jk})
- \frac{1}{2} \sum_{i,j} H_{ij} \partial_{ij} \,.
\end{equation}
Finally, using $\sum_{j} H_{ij} = \sum_{j} H_{ji} = 0$ for $H \in \cal M$,
we observe that $L$ from \eqref{e:gen1} can be rewritten as \eqref{e:L-cDBM}.
\end{proof}

\subsection{Outline of proof of Theorems~\ref{thm:gap}--\ref{thm:corr}}

Theorems~\ref{thm:gap}--\ref{thm:corr}
are an immediate consequence of the following two propositions.
As in \cite{1503.08702}, we set
\begin{equation}
  \label{e:D-UM}
  D \;\deq\; d \wedge \frac{N^2}{d^3}
  \,.
\end{equation}
We always assume $d \in [N^\alpha, N^{2/3-\alpha}]$, which implies $D \geq N^\alpha$.
To state the two propositions concisely, we introduce the following definition.
It will also be convenient in the proofs.

\begin{definition}
Given $H \in \cal M$, we denote by $\lambda_1 \geq \cdots \geq \lambda_{N - 1}$ the eigenvalues of $H \vert_{\f e^\perp}$.
Consider two random matrix ensembles $H_1$ and $H_2$ in $\cal M$. Then we say that
\begin{enumerate}
\item
the \emph{bulk eigenvalue gap statistics of $H_1$ and $H_2$ coincide} if 
for any $n \in \N$, $\phi \in C^\infty_c(\R^n)$, and $\kappa > 0$, we have
\begin{equation} \label{e:gap2}
\pb{\E_{H_1} - \E_{H_2}} \, \phi\pb{N\varrho(\gamma_i)(\lambda_i-\lambda_{i+1}), \dots, N\varrho(\gamma_i)(\lambda_i-\lambda_{i+n})} \;=\; o(1)
\end{equation}
as $N \to \infty$, uniformly in $i \in \qq{\kappa N, (1-\kappa)N}$;
\item
the \emph{averaged bulk eigenvalue correlation functions of $H_1$ and $H_2$ coincide}
if for any $n \in \N$, $\phi \in C^\infty_c(\R^n)$, $c > 0$ small enough, and $E \in (-2,2)$, we have for $b\deq N^{-1+c}$
\begin{equation} \label{e:corr2}
  \frac{1}{2b} \int_{E-b}^{E+b} \dd E'  \int_{\R^n} \phi(x_1, \dots, x_n) 
N^n \pb{p_{H_1}^{(n)} - p_{H_2}^{(n)}} \pbb{ E'+\frac{\dd x_1}{N\varrho(E)}, \dots, E'+\frac{\dd x_n}{N\varrho(E)} }
\;=\; o(1)\,,
\end{equation}
where the correlation functions $p_{H_i}^{(n)}$ are defined as in \eqref{e:def_corr_func}.
\end{enumerate}
Moreover, we say that the \emph{bulk eigenvalue statistics of $H_1$ and $H_2$ coincide} if (i) and (ii) hold.
\end{definition}

\begin{proposition} \label{prop:comp}
For any fixed $\delta > 0$ and $t \leq N^{-1-\delta}D^{1/2}$, the bulk eigenvalue statistics of $H(0)$ and $H(t)$ coincide.
\end{proposition}

\begin{proposition} \label{prop:univ}
For any fixed $\delta > 0$ and $t \geq N^{-1+\delta}$, the bulk eigenvalue statistics of $H(t)$ and $H(\infty) \eqdist W$ coincide.
\end{proposition}

Propositions~\ref{prop:comp}--\ref{prop:univ} are proved in Section~\ref{sec:pf}.
As mentioned in Section~\ref{sec:intro}, our main effort and novelty is 
in proving Proposition~\ref{prop:comp}.
Proposition~\ref{prop:univ} is essentially a consequence of general results on universality
of local eigenvalue statistics with small Gaussian component \cite{1504.03605}.
The local semicircle law of \cite{1503.08702} is an important input in the proofs of both propositions.

\begin{proof}[Proof of Theorem~\ref{thm:corr}]
  The proof is immediate from Propositions~\ref{prop:comp}--\ref{prop:univ},
  with $\delta \leq \alpha/4$.
\end{proof}

\section{Switchings and short-time comparision}
\label{sec:switch}

The main result of this section is Proposition~\ref{prop:EF0Ft} below.
To state it, we introduce the following Sobolev-type seminorms,
whereby the derivatives are taken in the directions of all switchings
\begin{equation} \label{e:calXdef}
  \cal X \;\deq\; \hb{\xi_{ij}^{kl} \in \R^{N\times N} :  i,j,k,l \in \qq{1,N}} \,.
\end{equation}
First, for $r \geq 1$, we define an $L^r$-seminorm on $\cal C^0(\cal M)$ through
\begin{equation} \label{e:normt}
  \|F\|_{r,t} \;\deq\; \pb{ \E |F(H(t))|^r }^{1/r} \,.
\end{equation}
Then, we extend this seminorm to include derivatives: for $F \in \cal C^n(\cal M)$ we define
\begin{equation} \label{e:normpart}
  \|\partial^n F\|_{r,t}
  \;\deq\;
  \normbb{\sup_{\theta \in [0,1]^n} \sup_{X \in \cal X^n}
    \absb{\partial_{X_1} \cdots \partial_{X_n} F\p{\,\cdot\,+(d-1)^{-1/2} \, \theta \cdot X)}}}_{r,t}\,,
\end{equation}
where $\partial_Y$ denotes the directional derivative in the direction $Y$, and for $\theta \in [0,1]^n$ and $X \in \cal X^n$ we abbreviate
\begin{equation*}
  \theta \cdot X\;\deq\; \theta_1 X_1 + \cdots + \theta_n X_n \,.
\end{equation*}

\begin{proposition} \label{prop:EF0Ft}
Let $H(t)$ be the constrained Dyson Brownian motion from Definition~\ref{def:constr_DBM}
with initial condition \eqref{e:Hdef}. Fix $\varepsilon > 0$ and let $r \equiv r(\epsilon)$ be large enough depending on $\epsilon$.
Then for any $F \in \cal C^4(\cal M)$ we have
\begin{equation} \label{e:EF0Ft}
  \E F(H(t)) - \E F(H(0))
  \;=\; O\pBB{D^{-1/2}N^{1+\varepsilon} \max_{1\leq i \leq 4}  \int_0^t \norm{\partial^i F}_{r,s} \, \dd s }
  \,.
\end{equation}
\end{proposition}

In the applications in Section~\ref{sec:pf}, we will use
functions $F$ satisfying $\|\partial^i F\|_{r,s} \leq N^c$ for $i \leq 4$ and a constant $c > 0$ that can be chosen arbitrarily small.
Thus, for $t\leq N^{-1 - \delta}D^{1/2}$ the right-hand side of \eqref{e:EF0Ft} 
will be $O(N^{-\delta + \epsilon + c})$ which is $o(1)$ provided that $c+\epsilon < \delta$.

The starting point for the proof of Proposition \ref{prop:EF0Ft} is the idea of \cite[Lemma~A.1]{BY2016},
namely to estimate the left-hand side of \eqref{e:EF0Ft} 
by estimating $\E(LF(H(t)))$.
However, since the entries of $H(t)$ are not independent, a different approach from \cite{BY2016} is needed
to control $\E (LF(H(t)))$. We do this by approximating the constrained DBM by a Markovian jump process induced by switchings.
This process is defined as follows.

\subsection{Switching dynamics}

We introduce a Markovian jump process on simple regular graphs by defining its generator
\begin{equation}
  Q f(A) \;\deq\; \frac{1}{8Nd} \sum_{i,j,m,n} I_{ij}^{mn}(A) \pB{ f(A-\xi_{ij}^{mn})-f(A) } \,,
\end{equation}
where we recall the definition of a switching from \eqref{e:Xijkl-def} and introduce the indicator function
\begin{equation} \label{e:def_Iijmn}
  I_{ij}^{mn}(A) \;\deq\; A_{ij}A_{mn}(1-A_{im})(1-A_{in})(1-A_{jm})(1-A_{jn})\,.
\end{equation}
The indicator function $ I_{ij}^{mn}(A)$ ensures that the graph encoded by $A$ contains the edges $\{i,j\}$ and $\{m,n\}$ but no other edges
between the four vertices $\{i,j,m,n\}$ (i.e.\ its restriction to $\{i,j,m,n\}$ is $1$-regular).

Thus, the process generated by $Q$ is a Markovian jump process whose
jump times are the events of a Poisson clock with a constant rate; at
each event of the clock, four vertices are selected uniformly at
random, and a switching as in Figure~\ref{fig:switch1} is performed on
the graph if the four vertices are connected by exactly two edges. It is
not hard to show that the uniform measure on $d$-regular graphs is
invariant under this jump process.

\begin{proposition} \label{prop:Jinv}
The uniform measure on simple $d$-regular graphs is invariant under $Q$.
This means that for any function $f$ on the set of simple $d$-regular graphs we have $\E(Qf(A)) = 0$.
\end{proposition}

The proof of the proposition is given in Section~\ref{sec:switch2}, in a slightly more general context.
The following proposition shows that the switching jump process generated by $Q$ is well approximated by the constrained DBM generated by $L$.

The generator $L$ acts naturally on functions of $H$ (denoted henceforth by an uppercase $F$), and the generator $Q$ on functions of $A$ (denoted henceforth by a lowercase $f$).
It is therefore convenient to introduce, for any $F \in \cal C^n(\cal M)$, the abbreviations
\begin{equation} \label{e:f_F_conv}
H \;=\; H_A \;\deq\; \frac{1}{\sqrt{d - 1}} (A - d \f e \f e^*) \,, \qquad f(A) \;=\; f_F(A) \;\deq\; F(H_A)\,.
\end{equation}

\begin{proposition} \label{prop:JLapprox}
For any $F \in \cal C^4 (\cal M)$ and using the notation \eqref{e:f_F_conv} we have 
\begin{equation} \label{e:JLapprox}
  Q f(A)
  \;=\; LF(H) + R \,,
\end{equation}
where
\begin{equation}
  \E R
  \;=\;
  O(D^{-1/2}N^{1+\varepsilon}) \max_{1 \leq i \leq 4} \norm{\partial^i F}_{r,0}
  \,.
\end{equation}
Here $\E$ denotes expectation with respect to the uniform measure on random $d$-regular graphs $A$.
\end{proposition}

The proof of this proposition is also deferred to Section~\ref{sec:switch2} below.
Roughly, the idea of the proof is as follows. By Taylor expansion, we obtain
\begin{equation} \label{e:Jtay-sketch}
  Q f(A) \;\approx\; \frac{1}{8Nd} \sum_{i,j,m,n} A_{ij}A_{mn} \pbb{ - \partial_{ij}^{mn} f(A) + \frac{1}{2} (\partial_{ij}^{mn})^2f(A) }
\end{equation}
with high probability.
Now $\E A_{ij} = \frac{d}{N}$ if $i \neq j$. By expanding $A_{ij}A_{mn} = (\frac{d}{N}+(A_{ij}-\frac{d}{N}))(\frac{d}{N}+(A_{mn}-\frac{d}{N}))$,
and keeping only the leading terms, we find that the right-hand side of \eqref{e:Jtay-sketch} becomes
by $LF(H$).
Here, for the second-order term on the right-hand side of \eqref{e:Jtay-sketch},
the leading term from $A_{ij}A_{mn}$ is $\frac{d^2}{N^2}$; for the first-order
term on the right-hand side of \eqref{e:Jtay-sketch},
the leading term from $A_{ij}A_{mn}$ is $\frac{d}{N} (A_{ij}-\frac{d}{N}) + \frac{d}{N} (A_{mn}-\frac{d}{N})$. 
Further error terms result from the dependence of the entries of the adjacency matrix.

Before giving the proofs of Propositions~\ref{prop:Jinv}--\ref{prop:JLapprox},
we deduce Proposition~\ref{prop:EF0Ft} from them.

\begin{proof}[Proof of Proposition~\ref{prop:EF0Ft}]
By \eqref{e:L_is_generator}, it suffices to estimate $\E [LF(H(t))]$.
By explicit solution of the constrained DBM, $H(t)$,
we find for any fixed $t \geq 0$ that
\begin{equation} \label{e:Ht_dist}
H(t) \;\eqdist\; \ee^{-t/2} H(0) + (1-\ee^{-t})^{1/2} W
\end{equation}
where $W$ is a copy of the constrained GOE independent of $H(0)$.
For the remainder of the proof,
we identify the right-hand side with $H(t)$, abbreviate $H \equiv H(0)$, and introduce the two functions
\begin{equation*}
F_W(H) \;=\; F_H(W) \;\deq\; F\pb{\ee^{-t/2} H + (1-\ee^{-t})^{1/2} W}\,,
\end{equation*}
where the choice of the argument determines the variables on which the generator $L$ acts. We recall the generator $L$ from \eqref{e:L-cDBM}, 
\begin{equation*}
L \;=\; \frac{1}{16N^3} \sum_{i,j,k,l} (\partial_{ij}^{kl})^2 - \frac{1}{32 N^2} \sum_{i,j,k,l} H_{ij}^{kl} \partial_{ij}^{kl}
\,.
\end{equation*}
From $\partial^2 = (\ee^{-t}+(1-\ee^{-t}))\partial^2$, 
$\ee^{-t/2} \partial F = \partial F_W$, and $(1 - \ee^{-t})^{1/2} \partial F = \partial F_H$,
we then deduce that $L F\pb{\ee^{-t/2} H + (1-\ee^{-t})^{1/2} W} = L F_W(H) + L F_H(W)$. We therefore get
\begin{equation*}
\E [LF(H(t))] \;=\; \E [L F_W(H)] + \E [L F_H(W)] \;=\; \E [L F_W(H)]\,,
\end{equation*}
where in the second step we used that the constrained GOE, $W$, is invariant with respect to the generator $L$.

Next, we define $f_W(A) \deq F_W(H)$ where $H \equiv H_A$ is defined as \eqref{e:f_F_conv}.
By Proposition~\ref{prop:Jinv}, the random $d$-regular graph $A$ is invariant with respect to the generator $Q$,
and Proposition~\ref{prop:JLapprox} therefore yields
\begin{equation*}
\E [L F_W(H)]
\;=\; \E[Q f_W(A)] +  O(D^{-1/2}N^{1+\varepsilon}) \max_{1 \leq i \leq 4} \norm{\partial^i F_W}_{r,0}
\;=\; O(D^{-1/2}N^{1+\varepsilon}) \max_{1 \leq i \leq 4} \norm{\partial^i F}_{r,t}
\,.
\end{equation*}
Thus, with \eqref{e:L_is_generator}, we have shown that
\begin{equation*}
\frac{\dd}{\dd t} \E[F(H(t))]
\;=\;
O(D^{-1/2}N^{1+\varepsilon}) \max_{1 \leq i \leq 4} \norm{\partial^i F}_{r,t}
\,,
\end{equation*}
and the claim follows by integrating over $t$.
\end{proof}

\subsection{Proofs of Propositions~\ref{prop:Jinv}--\ref{prop:JLapprox}}
\label{sec:switch2}

Propositions~\ref{prop:Jinv}--\ref{prop:JLapprox} concern switchings of regular graphs.
Switchings played an important role in the proof of the local semicircle law for
random regular graphs \cite{1503.08702}.
Here we use simple switchings instead of the double switchings needed in \cite{1503.08702}.

\begin{figure}[t]
\begin{center}
\input{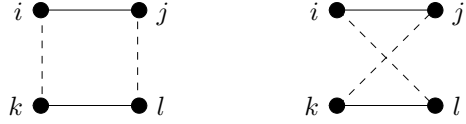}
\end{center}
\caption{Given four vertices $i,j,k,l$ with two edges between them,
there are two possible switchings. By equipping the edges with directions,
one of these two switchings can be selected canonically.\label{fig:switch}}
\end{figure}

Given two disjoint edges of a regular graph such that the graph has no other edges between
the vertices incident to these two edges, there are two possible switchings; see Figure~\ref{fig:switch}.
To specify one of these two switchings, it is convenient to assign to each of the edges to be switched a direction;
there is then a canonical choice between the two possible switchings. We write $ij$ for the edge $\{i,j\}$ directed from $i$ to $j$.

We consider sets $S$ of two directed edges of the complete graph, which we write in the form $S = \{ij, kl\}$.
We denote by $[S] = \{i,j,k,l\}$ the set of vertices incident to the edges of $S$.
For two such sets $S$ and $S'$, we define the indicator functions
\begin{align}
  \label{e:Idef}
  I(S) \;\equiv\; I(S;A) &\;\deq\; \ind{|[S]|=4 \text{ and } E|_{[S]} \text{ is $1$-regular}} \,,
  \\
  \label{e:Jdef}
  J(S,S') \;\equiv\; J(S,S';A) &\;\deq\; \ind{[S] \cap [S'] = \emptyset} \,,
\end{align}
where $E \equiv E(A) \deq \h{\h{i,j}: A_{ij}=1}$ is the set of (undirected) edges of the graph encoded by $A$,
and $E|_B \deq \h{e \in E: e \subset B}$ is the restriction of the graph $E$ to the subset of vertices $B$.
The indicator functions are illustrated in Figure~\ref{fig:simpswitchIJ}.
Note that $I_{ij}^{mn} = A_{ij}A_{mn}I(\{ij,mn\})$ (recall \eqref{e:def_Iijmn}).

\begin{remark}
The definitions \eqref{e:Idef}--\eqref{e:Jdef}
are similar to those given in
\cite[Section~6]{1503.08702}, with the following differences. First,
the current set $S$ consists of two directed edges instead of the
three undirected edges in \cite{1503.08702}. Because of the directions contained in the current set $S$,
it effectively incorporates the extra parameter $s$ of
\cite[Section~6]{1503.08702}. Second, the edges in $S$ are edges of
the complete graph, and we do not assume that they are contained in
some regular graph $A$; we will ultimately define the switching
associated with the set $S$ to act trivially unless $S$ is contained
in the edges $E$ of the given graph.
\end{remark}

For a set $S = \{ij,kl\}$ of two directed edges, we define the \emph{switching}
\begin{equation} \label{e:def_switching}
  T_{S}(A)
  \;\deq\; \begin{cases}
    A-\xi_{ij}^{kl} &  \text{if } I(S)=1, A_{ij}=1, A_{kl}=1\\
    A+\xi_{ij}^{kl} & \text{if } I(S)=1, A_{ik}=1, A_{jl}=1\\
    A & \text{otherwise}\,,
  \end{cases}
\end{equation}
where we recall the definition of $\xi_{ij}^{kl}$ from \eqref{e:Xijkl-def}.
In words, $T_S(A)$ switches the edges $S$ if they are contained in $A$ and are switchable
in the sense that the switching results again in a $d$-regular graph.
Moreover, for $S,S'$ as above, we define
\begin{align}
  T_{S,S'}(A)
  \;\deq\; \begin{cases}
  T_{S'}(T_S(A))
    & \text{if } J(S,S')=1\\
    A & \text{otherwise} \,.
  \end{cases}
\end{align}
In words, $T_{S,S'}(A)$ switches the edges in $S$ and $S'$ if they are contained in $A$ and the two switchings do not interfere with each other.

\begin{figure}
\begin{center}
\input{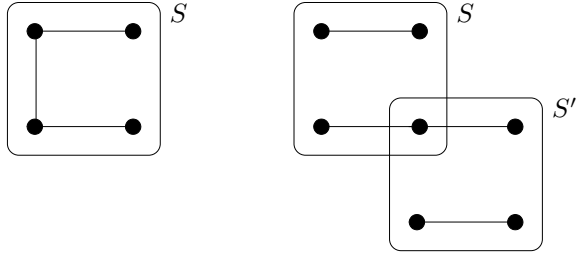}
\end{center}
\caption{
In the left diagram, $I(S)=0$ since the restricted graph is not $1$-regular.
In the right diagram, $J(S,S')=0$ since the two sets of vertices intersect.
\label{fig:simpswitchIJ}
}
\end{figure}

\begin{lemma} \label{lem:switchinv}
For any fixed $S,S'$ we have $A \eqdist T_S(A)$ and $A \eqdist T_{S,S'}(A)$.
\end{lemma}

\begin{proof}
It is easy to check that $T_S(A)$ is a $d$-regular graph if and only if $A$ is.
Moreover, $T_S(T_S(A)) = A$, so $T_S$ is a bijection on the set of $d$-regular graphs.
Since the distribution of $A$ is uniform, we obtain  $A \eqdist T_S(A)$.
The second claim follows similarly from $T_{S,S'}(T_{S,S'}(A))=A$.
\end{proof}

Now Proposition~\ref{prop:Jinv} follows easily.

\begin{proof}[Proof of Proposition~\ref{prop:Jinv}]
For any $f$, we get
\begin{align*}
  \sum_{i,j,m,n} \E\pb{ I_{ij}^{mn}(A) f(A) }
  &\;=\;
  \sum_{i,j,m,n} \E\pb{ A_{ij}A_{mn}I(\{ij,mn\};A) f(A) }
  \\
  &\;=\;
  \sum_{i,j,m,n} \E(A_{im}A_{jn}I(\{ij,mn\}; A + \xi_{ij}^{mn}) f(A+\xi_{ij}^{mn}))
  \\
  &\;=\;
  \sum_{i,j,m,n} \E\pb{A_{ij}A_{mn}I(\{ij,mn\}; A) f(A-\xi_{ij}^{mn})}
  \\
  &\;=\;
  \sum_{i,j,m,n} \E(I_{ij}^{mn}(A) f(A-\xi_{ij}^{mn}))
  \,,
\end{align*}
where the first and last steps follows from the definition of $I_{ij}^{mn}$,
the second step from Lemma~\ref{lem:switchinv},
and the third step from the exchangability of $i,j,m,n$ and using $I(S;A) = I(S;T_S(A))$.
This concludes the proof.
\end{proof}

For the proof of Proposition~\ref{prop:JLapprox}
we shall need estimates on the moments of entries of the adjacency matrix,
as well as estimates on such moments restricted to low-probability events
where the indicator functions \eqref{e:Idef}--\eqref{e:Jdef} are zero.
These estimates are collected in the following sequence of lemmas.

The following two lemmas show that moments of the entries of the adjacency matrix behave
roughly like those of an Erd\H{o}s-R\'enyi graph.

\begin{lemma} \label{lem:Amom}
Let $b \ll N$ and $i_1,j_1, \dots, i_b, j_b \in \qq{1,N}$.
Then for any $p \in \qq{1,N}$ and $q \in \qq{1,N} \setminus \{i_1,j_1, \dots, i_b,j_b\}$, we have
\begin{equation} \label{e:Amom}
  \E(A_{i_1j_1} \cdots A_{i_bj_b} A_{pq}) \;=\; O\pB{\frac{d}{N}} \E(A_{i_1j_1} \cdots A_{i_bj_b}) \,,
\end{equation}
where we use the convention $\E(A_{i_1j_1} \cdots A_{i_bj_b}) = 1$ if $b=0$.
\end{lemma}

\begin{proof}
Set $X \deq A_{i_1j_1}\cdots A_{i_bj_b}$ and $I \deq \{i_1,j_1, \dots, i_b, j_b, p\}$.
Then, since $\sum_{n} A_{pn} = d$ for any $p$, we find for any $q \notin I$ that
\begin{align*}
  \E(X)
  &\;=\; \frac{1}{d} \sum_{n}\E(X A_{pn}) 
  \;=\; \frac{1}{d} \sum_{n \not \in I}\E(X A_{pn}) 
  + \frac{1}{d} \sum_{n \in I}\E(X A_{pn})
  \\
  &\;\geq\; \frac{1}{d} \sum_{n \not \in I}\E(X A_{pn}) \;=\; \frac{N-|I|}{d} \E(X A_{pq})\,,
\end{align*}
where in the third step we used that $X A_{pn} \geq 0$ and in the last step that the law of $A$ is invariant under
permutation of vertices.
Using that $\abs{I} \leq N/2$ by assumption on $b$, the claim now follows.
\end{proof}

As a consequence of Lemma~\ref{lem:Amom}, we obtain the following explicit bounds.

\begin{lemma} \label{lem:Amom24}
Suppose that $|\h{i,j,m,n}| = 4-a$ and $|\h{i,j,k,l,m,n,p,q}| = 8-b$. Then
\begin{align} \label{e:Amom2}
  \E(A_{ij}A_{mn}) &\;=\; O\pB{\frac{d}{N}}^{2-\lfloor a/2\rfloor} \,,
  \\
  \label{e:Amom4}
  \E(A_{ij}A_{mn}A_{kl}A_{pq}) &\;=\; O\pB{\frac{d}{N}}^{4-\lfloor b/2 \rfloor} \,.
\end{align}
\end{lemma}

\begin{proof}
Since $A_{ss}=0$ for all $s$, we can assume that $i\neq j$, $m\neq n$, $k\neq l$, and $p\neq q$,
and thus $a \leq 2$ and $b \leq 4$. Then \eqref{e:Amom2}--\eqref{e:Amom4} follow easily from Lemma~\ref{lem:Amom}.
\end{proof}

In the next two lemmas, we estimate moments restricted to low-probability events
where the indicator functions \eqref{e:Idef}--\eqref{e:Jdef} vanish,
i.e.\ we estimate the contribution of graphs $A$ that are not switchable.
Throughout the rest of this section, for given indices $i,j,k,l,m,n,p,q$ we use the abbreviations
\begin{gather} \label{e:ind_fcts1}
  I_1 \;\deq\; I(\h{ij,mn};A) \,,\quad
  I_2 \;\deq\; I(\h{kl,pq};A) \,,
  \\
  \label{e:ind_fcts2}
  J_{12} \;\deq\; J(\h{ij,mn},\h{kl,pq};A) \,, \quad
  I_{12} \;\deq\; I_1I_2J_{12} \,,
\end{gather}
with $I$ and $J$ defined in \eqref{e:Idef}--\eqref{e:Jdef}.

\begin{lemma} \label{lem:indswitch}
Let $|\h{i,j,m,n}|=4-a$ and $|\h{i,j,k,l,m,n,p,q}|=8-b$. Then
\begin{align} \label{e:I1bd}
  \E((A_{ij}A_{mn}+A_{im}A_{jn})(1-I_1)) &\;=\; O\pB{\frac{d}{N}}^{3-a} \,.
  \\
  \label{e:IJbd}
  \E((A_{ij}A_{mn}+A_{im}A_{jn}) (A_{kl}A_{pq}+A_{kp}A_{lq}) (1-I_{12})) &\;=\; O\pB{\frac{d}{N}}^{5-b} \,.
\end{align}
\end{lemma}

\begin{proof}
First, assume that $i,j,k,l,m,n,p,q$ are all distinct, i.e.\ we consider the case $a=b=0$.
Then, since $|\h{i,j,m,n}|=4$ and $I_1 = 0$ implies that the graph $A$ restricted to $\{i,j,m,n\}$ is not $1$-regular,
we find
\begin{align*}
  \E(A_{ij}A_{mn}(1-I_1)) &\;\leq\; \E(A_{ij}A_{mn}(A_{im}+A_{in}+A_{jm}+A_{jn})) \,,
  \\
  \E(A_{im}A_{jn}(1-I_1))&\;\leq\; \E(A_{im}A_{jn}(A_{ij}+A_{mn}+A_{in}+A_{jm})) \,, 
\end{align*}
and Lemma~\ref{lem:Amom} implies that the right-hand sides are bounded by $O(d/N)^3$.
The proof of \eqref{e:IJbd} for $b = 0$ is analogous. We only consider the term $A_{ij}A_{kl}A_{mn}A_{pq}$;
the others dealt with similarly.
First, note that $J_{12} = 1$ if $|\h{i,j,k,l,m,n,p,q}|=8$.
Since $|\h{i,j,k,l,m,n,p,q}|=8$ and $I_1 I_2 = 0$ imply that $E \vert_{\h{i,j,m,n}}$ or $E \vert_{\h{k,l,p,q}}$
has at least three edges, we find
\begin{align*}
  &\E\pb{(A_{ij}A_{mn}A_{kl}A_{pq}(1-I_1I_2J_{12})}
  \;=\;
  \E\pb{(A_{ij}A_{mn}A_{kl}A_{pq}(1-I_1I_2)}
  \\
  \;\leq\;&
    \E\pb{(A_{ij}A_{mn}A_{kl}A_{pq}(A_{im} + A_{in} + A_{jm} + A_{jn} + A_{kp} + A_{kq} + A_{lp} + A_{lq})}
    \;=\; O\pB{\frac{d}{N}}^5\,,
\end{align*}
where the last step follows from Lemma \ref{lem:Amom}.

Finally, if $a>0$ we have $I_1=0$, and if $b>0$ we have $I_{12}=0$.
In these cases, we can directly apply \eqref{e:Amom2} and \eqref{e:Amom4}, respectively,
and the claim follows since $2- \floor{a/2} \geq 3 - a$ if $a>0$
and $4-\floor{b/2} \geq 5-b$ if $b>0$.
\end{proof}

As a consequence of Lemma~\ref{lem:indswitch}, we obtain the following averaged estimates.

\begin{lemma} \label{lem:indswitchsum}
If $|\h{i,j}|=2-a$ and $|\h{i,j,k,l}|=4-b$, then
\begin{align} \label{e:I1fixedbd}
  \frac{1}{N^2} \sum_{m,n} \E((A_{ij}A_{mn}+A_{im}A_{jn})(1-I_1)) &\;=\; O\pB{\frac{d}{N}}^{3-a} \,,
  \\
  \label{e:IJfixedbd}
  \frac{1}{N^4} \sum_{m,n}\sum_{p,q} \E((A_{ij}A_{mn}+A_{im}A_{jn}) (A_{kl}A_{pq}+A_{kp}A_{lq}) (1-I_{12})) &\;=\; O\pB{\frac{d}{N}}^{5-b} \,.
\end{align}
Moreover,
\begin{align} \label{e:I1sumbd}
  \frac{1}{N^4} \sum_{i,j,m,n} \E((A_{ij}A_{mn}+A_{im}A_{jn})(1-I_1)) &\;=\; O\pB{\frac{d}{N}}^3 \,,
  \\
  \label{e:IJsumbd}
  \frac{1}{N^8} \sum_{i,j,m,n}\sum_{k,l,p,q} \E((A_{ij}A_{mn}+A_{im}A_{jn}) (A_{kl}A_{pq}+A_{kp}A_{lq}) (1-I_{12})) &\;=\; O\pB{\frac{d}{N}}^{5} \,.
\end{align}
\end{lemma}

\begin{proof}
To prove \eqref{e:I1fixedbd}, we split the summation over $m,n$ by fixing $|\h{i,j,m,n}|=4-a-s$ where $s \in \qq{0,2}$; there are $O(N^{2-s})$ terms corresponding to each $s \in \qq{0,2}$.
By \eqref{e:I1bd}, the left-hand side of \eqref{e:I1fixedbd} is bounded by
\begin{equation*}
  O\pB{\frac{d}{N}}^{3-a}
  +
  \sum_{s=1}^2 O(N^{-s}) O\pB{\frac{d}{N}}^{3-a-s}
  \;=\; O\pB{\frac{d}{N}}^{3-a} \,.
\end{equation*}
The proofs of \eqref{e:IJfixedbd}--\eqref{e:IJsumbd} are analogous.
\end{proof}

Finally, as a consequence of Lemmas~\ref{lem:Amom}--\ref{lem:indswitchsum} and the H\"older inequality, we obtain the following estimates
incorporating an arbitrary function $f(A)$.
These and the remainder of the proof of Proposition~\ref{prop:JLapprox} are simplest to state in terms of versions of the seminorms
\eqref{e:normt}--\eqref{e:normpart} for $t=0$ without rescaling by $(d-1)^{-1/2}$.
Thus, instead of \eqref{e:normt} and \eqref{e:normpart}, we use the seminorms
\begin{equation*}
  \|f\|_r \;\deq\; \pb{\E |f(A)|^r }^{1/r}
\end{equation*}
and
\begin{equation*}
  \|\partial^n f\|_{r}
  \;\deq\;
  \normbb{\sup_{\theta \in [0,1]^n}\sup_{X \in \cal X^n} \absb{\partial_{X_1} \cdots \partial_{X_n} f\p{\,\cdot\,+\theta \cdot X}}}_{r}\,.
\end{equation*}

\begin{lemma}
Fix $\varepsilon > 0$ and let $r \equiv r(\epsilon)$ be large enough depending on $\epsilon$.
Let $f \in \cal C^0(\cal M)$ satisfy $\|f\|_r \leq 1$. Then if $|\h{i,j}|=2-a$ and $|\h{i,j,k,l}| = 4-b$, we have
\begin{align}
  \label{e:AAFbd}
  \frac{1}{N^2} \sum_{m,n} \E\pb{A_{ij}A_{mn}f(A)}
  &\;=\; O\pB{\frac{d}{N}}^{2- \lfloor a/2 \rfloor - \varepsilon} 
  \,,
  \\
  \label{e:AAAAFbd}
  \frac{1}{N^4} \sum_{m,n,p,q} \E\pb{A_{ij}A_{mn}A_{kl}A_{pq}f(A)}
  &\;=\; O\pB{\frac{d}{N}}^{4 - \lfloor b/2 \rfloor - \varepsilon} 
  \,,
  \\
  \label{e:I1fixedFbd}
  \frac{1}{N^2} \sum_{m,n} \E\pb{(A_{ij}A_{mn} + A_{im}A_{jn})\bar I_1f(A)}
  &\;=\; O\pB{\frac{d}{N}}^{3-a-\varepsilon} 
  \,,
  \\
  \label{e:IJfixedFbd}
  \frac{1}{N^4}
  \sum_{m,n} \sum_{p,q} \E\pb{(A_{ij}A_{mn} + A_{im}A_{jn}) (A_{kl}A_{pq} + A_{kp}A_{lq}) \bar I_{12}f(A)}
  &\;=\; O\pB{\frac{d}{N}}^{5-b-\varepsilon} 
  \,,
  \\
  \label{e:I1sumFbd}
  \frac{1}{N^4} \sum_{i,j,m,n} \E\pb{(A_{ij}A_{mn} + A_{im}A_{jn})\bar I_1f(A)}
  &\;=\; O\pB{\frac{d}{N}}^{3-\varepsilon} 
  \,,
  \\
  \label{e:IJsumFbd}
  \frac{1}{N^8}
  \sum_{i,j,m,n} \sum_{k,l,p,q} \E\pb{(A_{ij}A_{mn} + A_{im}A_{jn}) (A_{kl}A_{pq} + A_{kp}A_{lq}) \bar I_{12}f(A)}
  &\;=\; O\pB{\frac{d}{N}}^{5-\varepsilon} 
  \,,
\end{align}
where $\bar I_1 \deq 1-I_1$, $\bar I_{12} \deq 1-I_{12}$, and
the indicator functions $I_1$ and $I_{12}$ were defined in \eqref{e:ind_fcts1}--\eqref{e:ind_fcts2}.
\end{lemma}

\begin{proof}
We only prove \eqref{e:I1sumFbd}; the other estimates are proved similarly and we comment on
the differences at the end of the proof.
By H\"older's inequality, applied twice, first to $\E(\cdot)$ and then to the sum over $m,n$,
we obtain from \eqref{e:I1sumbd} that
\begin{align*}
&\mspace{-30mu}  \frac{1}{N^4} \sum_{i,j,m,n} \E\pb{(A_{ij}A_{mn} + A_{im}A_{jn})(1-I_1)f(A)}
\notag \\
  &\;\leq\;
  \frac{1}{N^4} \sum_{i,j,m,n} \qb{\E\pb{(A_{ij}A_{mn}+A_{im}A_{jn})(1-I_1)}}^{1 - 1/r} \|f\|_r
  \nonumber\\
  &\;\leq\;
  \pbb{ \frac{1}{N^4} \sum_{i,j,m,n} \E((A_{ij}A_{mn}+A_{im}A_{jn})(1-I_1))}^{1 - 1/r}  \|f\|_r
  \nonumber\\
  &\;\leq\;
  O\pB{ \frac{d}{N} }^{3 - 3/r}  \|f\|_r
  \;=\; O\pB{\frac{d}{N}}^{3-\varepsilon} \|f\|_r\,,
\end{align*}
where we chose $r$ large enough that $3/r \leq \epsilon$.

To prove \eqref{e:IJsumFbd}, we use \eqref{e:IJsumbd} instead of \eqref{e:I1sumbd},
and to prove \eqref{e:AAFbd}--\eqref{e:AAAAFbd} we apply \eqref{e:Amom} instead of \eqref{e:I1bd}.
To prove \eqref{e:I1fixedFbd}--\eqref{e:IJfixedFbd}, we use
\eqref{e:I1fixedbd}--\eqref{e:IJfixedbd}.
This concludes the proof.
\end{proof}

The next lemma estimates the effect of replacing $A_{ij}$ by its mean $d/N$,
or, equivalently, of conditioning on $\{A_{ij}=1\}$.
Since the entries of $A$ are not independent, we use switchings to analyse
such a conditioning.

\begin{lemma} \label{lem:condswitchbd}
Fix $\varepsilon > 0$ and let $r \equiv r(\epsilon)$ be large enough depending on $\epsilon$.
For any $f \in \cal C^2(\cal M)$ and
any $i,j,k,l$ with $|\h{i,j}|=2-a$ and $|\h{i,j,k,l}|=4-b$, we have
\begin{align} \label{e:condx1}
  \E\pB{f(A)\pB{A_{ij}-\frac{d}{N}}}
  &\;=\; O\pB{\frac{d}{N}}^{1-\varepsilon}\|\partial f\|_r  + O\pB{\frac{d}{N}}^{2-a-\varepsilon} \|f\|_r\,,
  \\
  \label{e:condx2}
  \E\pB{f(A) \pB{A_{ij}-\frac{d}{N}}\pB{A_{kl}-\frac{d}{N}}} 
  &\;=\; O\pB{\frac{d}{N}}^{2-\varepsilon} \|\partial^2 f\|_r 
  + O\pB{\frac{d}{N}}^{3-b-\varepsilon} \|f\|_r\,.
\end{align}
\end{lemma}

\begin{proof}
We begin with \eqref{e:condx1}.
Since $A \in \cal M + d\f e \f e^*$, we have 
$\sum_{m,n}A_{mn} = Nd$ and $\sum_{m} A_{im} = \sum_{n} A_{jn} = d$ for all $i$,
and the left-hand side of \eqref{e:condx1} 
is therefore equal to
\begin{equation} \label{e:condswitch-pf1}
  \E\pa{f(A)\pa{A_{ij}-\frac{d}{N}}}
  \;=\; \frac{1}{Nd} \sum_{m,n} \E\pb{f(A)\p{A_{ij}A_{mn}-A_{im}A_{jn}}}
  \,.
\end{equation}
Using \eqref{e:I1fixedFbd}, using the notation from \eqref{e:ind_fcts1}, we therefore find
\begin{equation*} \label{e:condswitch-pf2}
  \E\pa{f(A)\pa{A_{ij}-\frac{d}{N}}}
  \;=\; \frac{1}{Nd} \sum_{m,n} \E\pb{f(A)\p{A_{ij}A_{mn}-A_{im}A_{jn}}I_1}
  + O\pB{\frac{d}{N}}^{2-a-\varepsilon} \|f\|_r
  \,.
\end{equation*}
Because of the indicator function $I_1$, the first term on the right-hand side vanishes unless $a=0$.
Therefore we may assume that $a=0$ when estimating it.
By Lemma~\ref{lem:switchinv}, and since $I_1(A) = I_1(T_S(A))$ with $S=\{ij,mn\}$, the first term on the right-hand side equals
\begin{equation} \label{e:condswitch-pf3}
  \frac{1}{Nd} \sum_{m,n} \E\pB{\pb{f(A)-f(A-\xi_{ij}^{mn})}A_{ij}A_{mn}I_1} 
  \,.
\end{equation}
The difference of the $f$'s is bounded in absolute value by $\sup_{\theta \in [0,1]} \sup_{X \in \cal X} |\partial_X f(A + \theta X)|$.
Hence, \eqref{e:AAFbd} implies that \eqref{e:condswitch-pf3} is bounded by
\begin{equation*}
  O\pB{\frac{d}{N}}^{1-\varepsilon} \|\partial f\|_r \,.
\end{equation*}
This concludes the proof of \eqref{e:condx1}.

The proof of \eqref{e:condx2} is similar. As in \eqref{e:condswitch-pf1}, we write
\begin{equation*}
  \pB{A_{ij}-\frac{d}{N}}  \pB{A_{kl}-\frac{d}{N}}
  \;=\; \frac{1}{(Nd)^2} \sum_{m,n,p,q}
    \p{A_{ij}A_{mn}-A_{im}A_{jn}}  \p{A_{kl}A_{pq}-A_{kp}A_{lq}}
    \,.
\end{equation*}
As above, we write $1 = I_{12} + (1 - I_{12})$ inside the expectation on the left-hand side of \eqref{e:condx2}.
The second term yields a contribution of order $O\p{\frac{d}{N}}^{3-b-\varepsilon} \|f\|_r$, by \eqref{e:IJfixedFbd}.
The first term is zero unless $b = 0$ because of the factor $J_{12}$ in $I_{12}$.
We may therefore assume that $b = 0$ for the estimate of the first term.
Using Lemma \ref{lem:switchinv}, as in \eqref{e:condswitch-pf3}, we find that the first term is equal to
\begin{equation} \label{e:IIJmain}
  \frac{1}{(Nd)^2} \sum_{m,n,p,q}
  \E\pB{\pb{f(A)-f(A-\xi_{ij}^{mn})-f(A-\xi_{kl}^{pq})+f(A-\xi_{ij}^{mn}-\xi_{kl}^{pq})}A_{ij}A_{mn}A_{kl}A_{pq} I_{12}} \,.
\end{equation}
The difference of the four $f$'s is bounded in absolute value by
\begin{equation*}
\sup_{\theta_1,\theta_2 \in [0,1]} \sup_{X_1, X_2 \in \cal X} \absb{\partial_{X_1} \partial_{X_2} f(A+\theta_1 X_1 + \theta_2 X_2)}\,.
\end{equation*}
By \eqref{e:AAAAFbd}, we therefore find that \eqref{e:IIJmain} is bounded in absolute value by
\begin{equation*}
   O\pB{\frac{d}{N}}^{2-\varepsilon} \|\partial^2 f\|_r 
   \,.
\end{equation*}
This concludes the proof.
\end{proof}

Finally, with the preparations provided by the previous lemmas, we now complete the proof of Proposition~\ref{prop:JLapprox}.

\begin{proof}[Proof of Proposition~\ref{prop:JLapprox}]
First note that $I_{ij}^{mn} = A_{ij}A_{mn} I_1$.
By Taylor expansion, and writing $I_1 = 1+(I_1-1)$, we therefore have
\begin{equation} \label{e:Jtay}
  Q f(A) \;=\; \frac{1}{8Nd} \sum_{i,j,m,n} A_{ij}A_{mn} \pB{ - \partial_{ij}^{mn} f(A) + \frac{1}{2} (\partial_{ij}^{mn})^2f(A)) }  + N^2(R_1 + R_2) \,,
\end{equation}
where
\begin{align*}
  R_1 &\;=\; O\pB{\frac{N}{d}} \frac{1}{N^4} \sum_{i,j,m,n} A_{ij}A_{mn}(1-I_1) \sup_{\theta \in[0,1]} \sup_{X \in \cal X} |\partial_X f(A+\theta X)| \,,
  \\
  R_2 &\;=\; O\pB{\frac{N}{d}} \frac{1}{N^4} \sum_{i,j,m,n} A_{ij}A_{mn}\sup_{\theta \in[0,1]^3} \sup_{X \in \cal X^3} |\partial_{X_1}\partial_{X_2}\partial_{X_3} f(A+\theta \cdot X)| \,.
\end{align*}
By \eqref{e:I1sumFbd} and \eqref{e:AAFbd}, respectively, the two error terms are estimated by
\begin{equation*}
  \E R_1 \;=\; O\pB{\frac{d}{N}}^{2-\varepsilon} \|\partial f\|_r \,,\quad
  \E R_2 \;=\; O\pB{\frac{d}{N}}^{1-\varepsilon} \|\partial^3 f\|_r \,.
\end{equation*}

Next, we estimate the main terms in \eqref{e:Jtay}, which we write as
\begin{equation} \label{e:Q-L_main}
\frac{1}{8Nd} \sum_{i,j,k,l} A_{ij}A_{kl} \pB{ - \partial_{ij}^{kl} f(A) + \frac{1}{2} (\partial_{ij}^{kl})^2f(A)) }\,.
\end{equation}
The idea is to write $A_{ij} = \frac{d}{N} + (A_{ij}-\frac{d}{N})$ and likewise for $A_{kl}$. For the second-order term in \eqref{e:Q-L_main},
the term obtained by selecting both factors $\frac{d}{N}$ yields the main contribution.
More precisely, we write
\begin{equation*}
  \frac{1}{16 Nd} \sum_{i,j,k,l} A_{ij}A_{kl} (\partial_{ij}^{kl})^2f(A)
  \;=\; \frac{d}{16 N^3} \sum_{i,j,k,l} (\partial_{ij}^{kl})^2f(A)
  + N^2 (R_3+R_4)
  \,,
\end{equation*}
where
\begin{align*}
R_3 &\;=\; \frac{N}{8d} \frac{1}{N^4} \sum_{i,j,k,l} \pbb{\pb{(\partial_{ij}^{kl})^2f(A)} \pbb{A_{ij} - \frac{d}{N}} \frac{d}{N}}
\,,
\\
R_4 &\;=\; \frac{N}{16d} \frac{1}{N^4} \sum_{i,j,k,l} \pbb{\pb{(\partial_{ij}^{kl})^2f(A)} \pbb{A_{ij} - \frac{d}{N}} \pbb{A_{kl} - \frac{d}{N}}}\,.
\end{align*}
By \eqref{e:condx1} and \eqref{e:condx2}, respectively,
with $f$ replaced by $(\partial_{ij}^{kl})^2f$, we obtain
\begin{equation*}
  \E (R_3+R_4)
  \;=\; O\pB{\frac{d}{N}}^{1-\varepsilon}\pB{ \|\partial^3 f\|_r+\|\partial^4 f\|_r }
  + O\pB{\frac{d}{N}}^{2-\varepsilon} \|\partial^2 f\|_r  \,.
\end{equation*}

Next, we estimate the first-order term in \eqref{e:Q-L_main} using a similar argument. Here the term obtained by selecting
both factors $\frac{d}{N}$ from $A_{ij}$ and $A_{kl}$ vanishes
because $\sum_{i,j,k,l} \partial_{ij}^{kl} =0$;
the main contribution is given by the mixed term.
More precisely, we write
\begin{align*}
  \frac{1}{8Nd} \sum_{i,j,k,l} A_{ij}A_{kl} \partial_{ij}^{kl}f(A) 
  &\;=\;
  \frac{d}{8N^3} \sum_{i,j,k,l} \partial_{ij}^{kl}f(A)
  +
  \frac{1}{4N^2} \sum_{i,j,k,l} \pbb{A_{ij}-\frac{d}{N}} \partial_{ij}^{kl}f(A)
  + N^2 R_5
  \nonumber\\
  &\;=\;
  \frac{\sqrt{d-1}}{4N^2} \sum_{i,j,k,l} H_{ij} \partial_{ij}^{kl} f(A)
  + N^2 R_5
  \nonumber\\
  &\;=\;
  \frac{\sqrt{d-1}}{32 N^2} \sum_{i,j,k,l} H_{ij}^{kl} \partial_{ij}^{kl}f(A)
  + N^2 R_5
  \,,
\end{align*}
where
\begin{equation*}
  R_5 \;=\; \frac{N}{8d} \frac{1}{N^4} \sum_{i,j,k,l} \pbb{\pb{\partial_{ij}^{kl}f(A)} \pbb{A_{ij} - \frac{d}{N}} \pbb{A_{kl} - \frac{d}{N}}}
  \,.
\end{equation*}
By \eqref{e:condx2}, with $f$ replaced by $\partial_{ij}^{kl}f$, we obtain
\begin{equation*}
  \E R_5
  \;=\; O\pB{\frac{d}{N}}^{1-\varepsilon} \|\partial^3 f\|_r
  + O\pB{\frac{d}{N}}^{2-\varepsilon} \|\partial f\|_r\,.
\end{equation*}

We conclude that
\begin{align*}
Q f(A) &\;=\; \frac{d}{16 N^3} \sum_{i,j,k,l} (\partial_{ij}^{kl})^2f(A) - \frac{\sqrt{d-1}}{32 N^2} \sum_{i,j,k,l} H_{ij}^{kl} \partial_{ij}^{kl}f(A) + N^2 \sum_{i = 1}^5 R_i
\\
&\;=\; \frac{d - 1}{16 N^3} \sum_{i,j,k,l} (\partial_{ij}^{kl})^2f(A) - \frac{\sqrt{d-1}}{32 N^2} \sum_{i,j,k,l} H_{ij}^{kl} \partial_{ij}^{kl}f(A) + N^2 \sum_{i = 1}^6 R_i\,,
\end{align*}
where we defined
\begin{equation*}
R_6 \;\deq\; \frac{1}{16 N} \frac{1}{N^4} \sum_{i,j,k,l} (\partial_{ij}^{kl})^2f(A)\,.
\end{equation*}
Clearly, $\E R_6 = O\pb{\frac{1}{N}} \norm{\partial^2 f}_r$.

Using the notations introduced in \eqref{e:f_F_conv}, we have $\sqrt{d - 1} \, \partial f(A) = \partial F(H)$. Hence we obtain \eqref{e:JLapprox} with $R \deq N^2 \sum_{i = 1}^6 R_i$. The error term $R$ is estimated, using the above estimates on $\E R_i$, as
\begin{align*}
  \E R &\;=\; 
  O(N^{2 + \epsilon})
  \qbb{
  \pB{\frac{d}{N}}^{2} \pb{\|\partial f\|_r + \|\partial^2 f\|_r} + \frac{1}{N} \|\partial^2 f\|_r 
  +
  \frac{d}{N} \pb{\|\partial^3 f\|_r + \|\partial^4 f\|_r}
  }
  \\
  &\;=\;
  O(D^{-1/2}N^{1 + \epsilon})
  \qbb{
  \|\partial F\|_{r,0} + D^{-1/2} \|\partial^2 F\|_{r,0}
  +
  \|\partial^3 F\|_{r,0} + D^{-1/2} \|\partial^4 F\|_{r,0}
  }\,,
\end{align*}
as claimed.
\end{proof}

\section{Stability of eigenvectors and eigenvalues}
\label{sec:deloc}

In this section we derive basic stability properties for the
eigenvalues and eigenvectors of the Dyson Brownian motion $H(t)$.
These allow us to deduce estimates on the eigenvalues and eigenvectors of $H(t)$,
assuming similar estimates have been proved for $H(0)$.

As discussed in Section~\ref{sec:cDBM}, we consider a general Dyson Brownian motion $H(t)$ on an $M$-dimensional Hilbert space $V$,
with normalization constant $N$ as in \eqref{e:inner_prod}.
For the usual DBM we have $N=M$, while for the constrained DBM we have $M=N-1$;
we always assume that $N$ and $M$ are comparable.
We denote by $\lambda_1(t) \geq \cdots \geq \lambda_{M}(t)$ the eigenvalues of $H(t)$,
and by ${\f v}_1(t), \dots, {\f v}_{M}(t) \in V$ the associated normalized eigenvectors of $H(t)$.
Moreover, we define the Stieltjes transform of the empirical spectral measure of $H(t)$
by $s(t;z) \deq \frac{1}{M} \sum_{i = 1}^{M} \frac{1}{\lambda_i(t) - z}$.

Throughout the rest of the paper, we use the following notion of high probability events
and high probability bounds, introduced in \cite{MR3119922}.

\begin{definition} \label{def:highprob}
\begin{enumerate}
\item
We say that an event $\Xi$ has \emph{high probability} if
for every $\zeta>0$ there exists an $N_0(\zeta) > 0$ such that $\P(\Xi^c) \leq N^{-\zeta}$ for $N \geq N_0(\zeta)$. 
\item
For nonnegative random variables $A,B$, we write $A \prec B$ or $A = O_\prec(B)$ if for any $\zeta>0$ there exists an $N_0(\zeta)$ such that
$\P(A > N^{1/\zeta}B) \leq N^{-\zeta}$ for $N \geq N_0(\zeta)$.
\end{enumerate}

\medskip
If the event $\Xi$ from (i) and the random variables $A$ and $B$ from (ii) depend on
some additional parameter $u \in U$
in some possibly $N$-dependent set $U$, we we say that (i) and (ii) hold uniformly in $u$
if $N_0(\zeta)$  does not depend on $u$.
\end{definition}

Throughout the following, the definitions (i) and (ii) will always be uniform in all parameters,
such as $z$, any matrix indices, and deterministic vectors.
Note that $\prec$ is compatible with the usual algebraic operations,
so that for instance we have $\sum_{i} A_i \prec \sum_i B_i$ provided that $A_i \prec B_i$ for 
all $i$ and the size of the index set for $i$ is $N^{O(1)}$.

\subsection{Delocalization of eigenvectors}

The following result shows that if all eigenvectors of $H(0)$ are uniformly delocalized
in some direction $\f q \in V$,
then with high probability they remain delocalized in this direction under the DBM on $V$,
for any time $t>0$.

\begin{lemma} \label{lem:evcont}
  Suppose that $H(t)$ is the DBM on an $M$-dimensional space $V$.
  Let $\f q \in V$ and suppose that $\max_{i} | \f q\cdot {\f v}_i(0)| \leq B$.
  Then, for any $t>0$, any $i \in \qq{1,M}$,
  and $\xi \gg 1$,
  \begin{equation}
    \P\pa{|\f q\cdot {\f v}_i(t)| \geq \xi B} \;\leq\; \ee^{-\frac12 \xi^2}\,.
  \end{equation}
  In particular,
  \begin{equation}
    |\f q\cdot {\f v}_i(t)| \;\prec\; B \,.
  \end{equation}
\end{lemma}

Lemma~\ref{lem:evcont} is a simple consequence of the
\emph{eigenvector moment flow} (EMF) introduced in \cite{BY2016}.
Suppose for simplicity that the eigenvalues of $H(0)$ are
distinct. Then the eigenvalue process $(\lambda_i(t))$ 
is almost surely continuous and simple for all $t > 0$; see
\cite{BY2016} for more details. We study the dynamics of the
eigenvectors ${\f v}_i(t)$ by conditioning on the eigenvalue process;
see again \cite{BY2016} for a precise construction. Hence, for the
following argument, we condition on $(\lambda_i(t))$ and regard the
eigenvalue process as deterministic.

We give the definition of the EMF restricted to moments of a fixed order $p \in \N$.
The configuration space is $\Omega_p \deq \hb{ \eta =(\eta_i)_{i = 1}^{M} \in \N^M \col \sum_{i = 1}^{M} \eta_i = p}$.
The configurations $\eta \in \Omega_p$ are interpreted as configurations of $p$ particles on the lattice $\qq{1,M}$,
whereby a single site of $\qq{1,M}$ may be occupied by multiple particles.
We denote by $\eta^{i,j} \deq \eta + \ind{\eta_i > 0}(\f e_j - \f e_i)$ the configuration obtained from $\eta$ by moving one particle from $i$ to $j$.
The time-dependent generator $R(t)$ of the EMF is defined by
\begin{equation*}
(R(t)f)(\eta) \;\deq\; \sum_{i\neq j} W_{ij}(t) 2\eta_i (1+2\eta_i) (f(\eta^{i,j}) - f(\eta)) \,,
\end{equation*}
where
\begin{equation*}
  W_{ij}(t) \;\deq\; \frac{1}{N(\lambda_i(t) -\lambda_j(t))^2} \,.
\end{equation*}

For our purposes, the precise form of the coefficients $W_{ij}(t)$ is not important;
we only use that they are nonnegative and continuous in $t$.
The $p$-particle EMF is given by the equation
\begin{equation} \label{e:emf}
  \partial_t f_t(\eta) \;=\; (R(t)f_t)(\eta) \,,\qquad f_0 : \Omega_p \to \R \text{ given}
  \,.
\end{equation}
This is a linear (time-dependent) ODE on a finite dimensional vector space, and thus well-posed.
It is also easy to see that it is contractive on $L^\infty(\Omega)$
in the sense that $\|f_t\|_{L^\infty(\Omega_p)} \leq \|f_0\|_{L^\infty(\Omega_p)}$.

Next, for deterministic $\eta \in \Omega_p$ and $\f q \in V$, we define
\begin{equation} \label{e:EMF-ft}
  f_t(\eta) \;\deq\;
  \E\qBB{\prod_{i = 1}^M \frac{1}{(2 \eta_i - 1)!!} \, \p{\f q \cdot \f v_i(t)}^{2 \eta_i} \;\Bigg\vert\; \pb{\lambda_i(t) \col i \in \qq{1,M}, t \geq 0}} \,,
\end{equation}
where $n !! \deq n \cdot (n - 2) \cdots 3 \cdot 1$ for odd $n$, and by convention $(-1)!! = 1$.
In \cite[Theorem 3.1]{BY2016} it is shown that $f_t$ solves \eqref{e:emf}.

\begin{remark}
In \cite{BY2016}, Dyson Brownian motion is defined without the Ornstein-Uhlenbeck drift term in the SDE \eqref{e:standard_DBM},
and the SDEs for the eigenvalues and eigenvectors are stated in \cite[Definition~2.2]{BY2016}.
In the present case, with drift term, the SDEs for eigenvalue and eigenvector flows
are given by
\begin{align*}
 \dd \lambda_i\;&=\; \frac{\dd B_{ii}}{\sqrt{N}}+\frac{1}{N}\sum_{j: j\neq i}\frac{1}{\lambda_i-\lambda_j}\dd t-\frac{\lambda_i}{2}\dd t\,, \\
 \dd \f v_i\;&=\; \frac{1}{\sqrt{N}}\sum_{j:j\neq i}\frac{\dd B_{ij}}{\lambda_i-\lambda_j}\f v_j- \frac{1}{2N}\sum_{j:j\neq i}\frac{\dd t}{(\lambda_i-\lambda_j)^2}\f v_i \,,
\end{align*}
for $i=1,2,\dots,M$, and with $B(t)$ a Brownian motion on the space of $M \times M$ real symmetric matrices with
quadratic covariation $\scalar{B_{ij}}{B_{kl}}(t) = (\delta_{ik} \delta_{jl} + \delta_{il} \delta_{jk}) t$.
Thus, the SDEs for the eigenvectors are the same with or without the drift term.
Therefore the arguments of \cite[Section~3]{BY2016} apply verbatim in our setting as well.
\end{remark}

\begin{proof}[Proof of Lemma~\ref{lem:evcont}]
Suppose first that $H(0)$ has simple spectrum.
Let $f_t$ be the given by \eqref{e:EMF-ft}, which solves \eqref{e:emf} as remarked above.
Then, since the EMF \eqref{e:emf} is a contraction on $L^\infty(\Omega_p)$,
we obtain from the assumption of Lemma \ref{lem:evcont} that
\begin{equation*}
  \max_{\eta \in \Omega_p} |f_t(\eta)| \;\leq\; \max_{\eta \in \Omega_p} |f_0(\eta)|
  \;\leq\; B^{2p}
  \,.
\end{equation*}
Therefore, choosing $\eta = p \, \f e_i$, we get
\begin{equation*}
  \E\qb{(\f q\cdot \f v_i(t))^{2p}} \;=\; (2p - 1)!! \, \E [f_t(\eta)] \;\leq\; (2p - 1)!! B^{2p}\,, 
\end{equation*}
from which the claim follows.
Finally, if $H(0)$ does not have simple spectrum, the same estimate holds by
a simple approximation argument using the continuity of the eigenvectors as functions of the matrix.
\end{proof}

\subsection{Stability of eigenvalues}

The following result shows that if the empirical spectral measure at $t=0$ is close to the semicircle law,
this remains true for $t>0$.
For its statement, recall that $s(t,z)$ denotes the Stieltjes transform of the empirical spectral measure of $H(t)$.
We denote the Stieltjes transform of the semicircle law by $m$.
It can be characterized as the unique holomorphic function $m: \C_+ \to \C_+$
such that $m^2+mz+1=0$ and $m(z) \sim 1/z$ as $|z|\to\infty$; see e.g.\ \cite{MR2760897}.

\begin{lemma} \label{lem:scont}
Suppose that $C^{-1} M \leq N \leq C M$. Fix $\epsilon > 0$. If for some $B \leq N^{-\epsilon}$ we have
\begin{align}\label{e:asup}
|s(0;z)-m(z)|
\;\prec\; B+\frac{1}{(N\eta)^{1/4}}
\end{align}
uniformly for $z=E+\ii\eta$ with $\eta\geq N^{-1+\varepsilon}$,
then for any $t\leq B$ we have
\begin{align}\label{e:sTrigid}
|s(t;z)-m(z)|
\;\prec\; B+\frac{1}{(N\eta)^{1/4}}
\,.
\end{align}
uniformly for $z=E+\ii\eta$ with $\eta \in [N^{-1+\varepsilon},1]$.
\end{lemma}

\begin{proof}
Define $s_{\text{fc},t}(z)$ as the unique solution $\C_+ \to \C_+$  of the self-consistent equation
\begin{equation}\label{e:mtdef}
  s_{\text{fc},t}(z) \;=\; \frac{1}{M} \sum_{i = 1}^{M} \frac{1}{\ee^{-t/2} \lambda_i(0) - z - (1-\ee^{-t}) s_{\text{fc},t}(z)}\,.
\end{equation}
Thus, $s_{\text{fc},t}(z)$ is the Stieltjes transform of the free convolution of
the empirical eigenvalue distribution of  $\ee^{-t/2} H(0)$  and the semicircle law rescaled by $(1 - \ee^{-t})^{1/2}$.
We refer to \cite{MR1488333} for the existence and uniqueness of $s_{\text{fc},t}(z)$
and relative properties on the free convolution with semicircle law.

As in \eqref{e:Ht_dist}, we find that $H(t) \eqdist \ee^{-t/2} H(0) + (1-\ee^{-t})^{1/2} W$,
where $W$ is the standard Gaussian measure on $\cal H(V)$ with inner product \eqref{e:inner_prod}.
Under the assumptions of the lemma, \cite[Corollary~7.11]{1504.03605} implies that for $t\leq N^{-\epsilon}$ we have
\begin{equation} \label{e:ssfc}
|s(t;z) - s_{\text{fc},t}(z)| \;\prec\; \frac{1}{(N\eta)^{1/3}}
\end{equation}
uniformly for $z=E+\ii\eta$ with $\eta\geq N^{-1+\varepsilon}$.
(Note that in \cite{1504.03605}, the Stieltjes transform is
denoted by $m_V$ instead of $s$, and that $s_{\text{fc},t}$ is denoted $m_{\text{fc},t}$.
Moreover, \cite[Corollary~7.11]{1504.03605} is stated for a diagonal matrix  $H(0)$; however, since
 $W$  is invariant under orthogonal transformations  which diagonalize $H(0)$, 
the results of \cite{1504.03605} trivially apply to any symmetric matrix $H(0)$.)

Set $\vartheta_t \deq 1-e^{-t}\leq t$.
Note that the Stieltjes transform of the empirical eigenvalue distribution of $\ee^{-t/2}H (0)$ is given by $\ee^{t/2}s(0,\ee^{t/2}z)$,
and that \eqref{e:mtdef} can be rephrased as
\begin{equation*}
s_{\text{fc},t}(z) \;=\; \ee^{t/2}s(0,\ee^{t/2}(z+\vartheta_t s_{\text{fc},t}(z))) \;.
\end{equation*}
For any $z=E+\ii\eta$ such that $\eta\geq N^{-1+\epsilon}$,
we have $\im \ee^{t/2}(z+\vartheta_t s_{\text{fc},t}(z))\geq \im \ee^{t/2}z \geq N^{-1+\epsilon}$,
where we used that $\im s_{\text{fc},t}(z)>0$. From the assumption \eqref{e:asup} we therefore get
\begin{align} \label{e:mbound}
s_{\text{fc},t}(z)
\;=\;
\ee^{t/2}m(\ee^{t/2}(z+\vartheta_t s_{\text{fc},t}(z))) +O_{\prec}\pbb{B+\frac{1}{(N\eta)^{1/4}}}
\,.
\end{align}
Next, note that 
\begin{align}\label{e:sf}
m(z)
\;=\;
\ee^{t/2}m(\ee^{t/2}(z+\vartheta_t m(z))) \,.
\end{align}
%
%
(This may be interpreted as the fact that the semicircle law with variance $t$ is a semigroup with respect to free convolution.)
Moreover, from the definition of $m(z)$ it is easy to deduce the continuity estimate 
\begin{equation} \label{e:mzwbd}
|m(z)-m(w)| \;\leq\; 2|z-w|^{1/2}\,,
\end{equation}
for any $z,w \in \C_+$.

By \eqref{e:mzwbd}, and using that $t= O(1)$, the difference between \eqref{e:mbound} and \eqref{e:sf} is
\begin{align*}
|s_{\text{fc},t}(z)-m(z)|
&\;=\;
\ee^{t/2}\left|m(\ee^{t/2}(z+\vartheta_t s_{\text{fc},t}(z)))-m(\ee^{t/2}(z+\vartheta_t m(z)))\right|+O_{\prec}\pbb{B+\frac{1}{(N\eta)^{1/4}}}\\
&\;\leq\;
O(t^{1/2}) \absb{s_{\text{fc},t}(z)-m(z)}^{1/2}+O_{\prec}\pbb{B+\frac{1}{(N\eta)^{1/4}}}\\
&\;\leq\; 
\max\left\{O(t^{1/2}) \absb{s_{\text{fc},t}(z)-m(z)}^{1/2}, O_{\prec}\pbb{B+\frac{1}{(N\eta)^{1/4}}}\right\}
\,.
\end{align*}
Therefore either $|s_{\text{fc},t}(z)-m(z)|= O(t)$ or $|s_{\text{fc},t}(z)-m(z)|\prec B+(N\eta)^{-1/4}$, and we get
\begin{equation} \label{e:sfcm}
|s_{\text{fc},t}(z)-m(z)|
\;\prec\; B+\frac{1}{(N\eta)^{1/4}}+t
\;\prec\; B+\frac{1}{(N\eta)^{1/4}}
\,,
\end{equation}
where we used $t \leq B$. Combining \eqref{e:ssfc} and \eqref{e:sfcm} and using $\eta \leq 1$, the claim \eqref{e:sTrigid} follows.
\end{proof}

\section{Proof of Propositions~\ref{prop:comp}--\ref{prop:univ}}
\label{sec:pf}

With the preparations provided by Sections~\ref{sec:switch}--\ref{sec:deloc},
and using results of \cite{1503.08702,1504.03605,MR3429490},
we now complete the proofs of Propositions~\ref{prop:comp}--\ref{prop:univ}.
First, recall that $\alpha > 0$ is fixed, and that we always assume $D \geq N^{\alpha}$.
We also use the notation $z = E + \ii \eta$ for the real
and imaginary parts of the spectral parameter $z \in \C_+$.

Throughout this section,
$H(t)$ denotes the constrained DBM from Definition \ref{def:constr_DBM}
with $H(0)$ given by \eqref{e:Hdef}.
We use the notations of Section~\ref{sec:deloc} applied to the constrained DBM.
In particular,
\begin{equation*}
  M \;\deq\; N-1
\end{equation*}
is the dimension of the space $V \deq \f e^\perp$.

\subsection{A priori estimates on eigenvalues and eigenvectors}

We begin by collecting some results
on the eigenvalues and eigenvectors  of $H(t)|_{\f e^\perp}$ required for
the proofs of Propositions~\ref{prop:comp}--\ref{prop:univ}.

For any $H \in \cal M$, we denote the eigenvalues of $H|_{\f e^\perp}$
by $\lambda_1(H)\geq \cdots \geq \lambda_{M}(H)$,
and the corresponding normalized eigenvectors by $\f v_1(H),\dots, \f v_{M}(H)$.
The components of the eigenvectors in the standard basis on $\R^N$ are denoted
$v_k(H;i) \deq \f e_i \cdot \f v_k(H)$, $i \in \qq{1,N}$, $k \in \qq{1,M}$.
Moreover, for $H \in \cal M$,
we denote by $G_{ij}(H;z)$ the entries of the Green's function of $H$ restricted to $\f e^\perp$
in the standard basis of $\R^N$,
and by $s(H;z)$ the Stieltjes transform of the empirical spectral measure.
Explicitly,
\begin{align} \label{e:Gdef-again}
  G_{ij}(H;z)
  &\;\deq\; 
  \sum_{k = 1}^M 
  \frac{v_k(H;i)v_k(H;j)}{\lambda_k(H) - z} \,,
  \\ \label{e:sdef}
  s(H;z)
  &\;\deq\;
  \frac{1}{M}
  \tr G(H;z)
  \;=\;
  \frac{1}{M} 
  \sum_{k = 1}^M 
  \frac{1}{\lambda_k(H) - z} \,.
\end{align}
Finally, we set
\begin{equation} \label{e:Gammadef-again}
  \Gamma(H) \;\equiv\;
  \Gamma(H;z) \;\deq\; \max_{i,j} \abs{G_{ij}(H;z)}  \vee 1\,.
\end{equation}
We also recall the definition of the typical location $\gamma_i$ of the $i$-th eigenvalue from \eqref{e:gammadef}.

The following proposition summarizes the input we need from the local semicircle law of \cite{1503.08702}.
The local semicircle law, as proved in \cite{1503.08702}, only applies for $t=0$,
and the extension to $t>0$ is provided by the results of Section~\ref{sec:deloc}.

\begin{proposition} \label{prop:Gbdt}
  For any $z\in\C_+$, $i \in \qq{1,N}$, $k \in \qq{1, M}$, and
  $0 \leq t \leq D^{-1/4}$, we have
  \begin{equation} \label{e:Gbdt}
    |v_k(H(t);i)| \;\prec\; N^{-1/2} \,,\qquad
    \Gamma(H(t);z) \;\prec\; 
    1+ \frac{1}{N\eta} \,.
  \end{equation}
  Moreover, for any fixed $\kappa > 0$ and any $i \in \qq{\kappa N, (1-\kappa)N}$, we also have
  \begin{equation} \label{e:rig}
    |\lambda_i(H(t))-\gamma_i| \;\prec\; D^{-1/4}\,.
  \end{equation}
\end{proposition}

\begin{proof}
First, as special cases of
\cite[Theorem~1.1 and Corollary~1.2]{1503.08702},
for any $z=E+\ii\eta$ with $E\in \R$ and $\eta \geq N^{-1+\epsilon}$,
for arbitrary $\epsilon>0$, we have
\begin{equation} \label{e:svbd0}
  |s(H(0);z)-m(z)| \;\prec\; \frac{1}{D^{1/4}}+\frac{1}{(N\eta)^{1/4}} \,,\qquad
  |v_k(H(0);i)| \;\prec\; N^{-1/2} \,.
\end{equation}
(Note that the local semicircle law from \cite{1503.08702} also includes the trivial eigenvalue at $0$;
it is easy to see that its contribution to $s$ is negligible compared to the error bounds in \eqref{e:svbd0}.)

Next, we extend these bounds from $t=0$ to $t >0$.
For $i \in \qq{1,N}$ define $\hat {\f e}_i = \f e_i - (\f e_i \cdot \f e) \f e\in \f e^\perp$.
Since $v_k(H(t);i) = \hat {\f e}_i \cdot \f v_k(H(t))$, 
from \eqref{e:svbd0} and Lemma~\ref{lem:evcont}, applied to the constrained DBM with $\f q = \hat {\f e}_i$, we find
$|v_k(H(t);i)| \prec N^{-1/2}$, for any $t>0$.
Similarly, for $t \leq D^{-1/4}$, the extension of the bound on the Stieltjes transform follows immediately from Lemma~\ref{lem:scont}
with $B=D^{-1/4}$.
Summarizing, for any $\eta \geq N^{-1+\epsilon}$ and $0\leq t \leq D^{-1/4}$, we have
\begin{equation} \label{e:svbdt}
  |s(H(t);z)-m(z)| \;\prec\; \frac{1}{D^{1/4}}+\frac{1}{(N\eta)^{1/4}}\,,\qquad
  |v_k(H(t);i)| \;\prec\; N^{-1/2} \,.
\end{equation}
This proves the first estimate of \eqref{e:Gbdt}. 

In order to prove the second estimate of \eqref{e:Gbdt},
we use a dyadic decomposition (see e.g.\ \cite[(8.2)]{MR2981427}) to obtain, for any matrix $H \in \cal M$,
\begin{equation*}
|G_{ij}(z)|
\;\leq\; \sum_{k = 1}^{M}\frac{|v_{k}(i)v_{k}(j)|}{|\lambda_k-E+\ii\eta|}
\;\leq\; 4 N\max_{k,l} |v_k(l)|^2\pBB{1+\sum_{n=0}^{\lceil\log_2\eta^{-1}\rceil}\im s(E+\ii2^n \eta)}.
\end{equation*}
We apply this estimate to the matrix $H(t)$.
By \eqref{e:svbdt}, we have $\max_{k,l} |v_k(l)|^2 \prec 1/N$.
Moreover, since $\eta \im s(E+\ii\eta)$ is increasing in $\eta$ (as may be easily seen from the right-hand side of \eqref{e:sdef}),
and since $|m|\leq 1$,
the first bound in \eqref{e:svbdt} implies
$\im s(z) \prec 1+1/(N\eta)$
for any $\eta > 0$, and thus $\im s(E+\ii2^n\eta) \prec 1 + 2^{-n}/N\eta$.
For $\eta \geq 1/N^{O(1)}$ we then have $\log \eta^{-1} \prec 1$ and obtain $\Gamma(z) \prec 1$ as desired.
For arbitrary $\eta >0$ the claim then follows by \cite[Lemma~2.1]{1503.08702}.
(In fact, we shall only need \eqref{e:Gbdt} with $\eta \geq 1/N^{O(1)}$.)

Finally, we deduce \eqref{e:rig} from the bound on the Stieltjes transform in \eqref{e:svbdt}.
We abbreviate $\lambda_k \equiv \lambda_k(H(t))$, and denote by
\begin{equation*}
\varrho(I) \;\deq\; \int_I \varrho(x) \, \dd x \,, \qquad \nu(I) \;\deq\;
\frac{1}{M} \sum_{k = 1}^{M} \ind{\lambda_k \in I}
\end{equation*}
the semicircle and empirical spectral measures, respectively, applied to an interval $I$.
Then, following a standard application of the Helffer-Sj\"ostrand functional calculus
along the lines of \cite[Section 8.1]{MR3098073},
we find from \eqref{e:svbdt} and $D \leq N$ that for any interval $I \subseteq [-3,3]$ we have
\begin{equation} \label{e:HS}
\abs{\nu(I) - \varrho(I)}
\;\prec\; \frac{1}{D^{1/4}}+\frac{1}{N^{1/4}}
\;\prec\; \frac{1}{D^{1/4}}
\,.
\end{equation}
(We note that previously \eqref{e:HS} for $t = 0$ was given in \cite[Corollary~1.3]{1503.08702}.)
Using \eqref{e:HS}, we may estimate $\lambda_i-\gamma_i$ as follows.
By \eqref{e:HS} applied to $I=[-3,3]$, we find that there are at most $O_\prec(N D^{-1/4})$ eigenvalues outside $[-3,3]$.
Defining $f(E) \deq \varrho([E, \infty))$, we therefore find from \eqref{e:gammadef} and \eqref{e:HS} that
\begin{align*}
f(\gamma_i) \;=\; \frac{i}{N} \;=\; \nu([\lambda_i, \infty)) + O \pB{\frac{1}{N}}
\;&=\; \nu([\lambda_i, 3)) + O_\prec \pB{\frac{1}{D^{1/4}}}
\\
&=\; \varrho([\lambda_i, 3)) + O_\prec \pB{\frac{1}{D^{1/4}}} \;=\; f(\lambda_i) + O_\prec \pB{\frac{1}{D^{1/4}}}\,.
\end{align*} 
Since $i \in \qq{\kappa N, (1-\kappa N)}$, we have $\abs{f'} \geq c > 0$ in a neighbourhood of $\gamma_i$,
and we therefore get \eqref{e:rig}.
This concludes the proof.
\end{proof}

The next result shows that the suprema in \eqref{e:normpart} do not essentially change the size of $\Gamma$.

\begin{corollary}\label{cor:pertG}
  Fix $n \in \N$.
  For any $z \in \C_+$ and $0 \leq t \leq D^{-1/4}$, we have 
  \begin{equation} \label{e:GammaXXbd}
    \sup_{\theta \in [0,1]^n} \sup_{X \in \cal X^n}
    \Gamma\pb{H(t)+(d-1)^{-1/2} \, \theta \cdot X;z} \;\prec\; 
    1+ \frac{1}{N\eta} \,.
  \end{equation}
  Moreover, for any $i \in \qq{1,N}$ and $k \in \qq{1, M}$,  we have
  \begin{equation} \label{e:vXXbd}
    \sup_{\theta \in [0,1]^n} \sup_{X \in \cal X^n}
    \absb{v_k\pb{H(t)+(d-1)^{-1/2}\, \theta \cdot X;i}} \;\prec\; N^{-1/2} \,.
  \end{equation}
\end{corollary}

\begin{proof}
We abbreviate $H \equiv H(t)$.
Without loss of generality, by an argument analogous to \cite[Lemma~2.1]{1503.08702},
we may assume that $\eta \geq 1/N$. Hence, by \eqref{e:Gbdt}, we have $\Gamma(H;z) \prec 1$.
It therefore suffices to show that if $\Gamma(H;z) \leq  (d-1)^{1/2}/(16n)$ then
for any $\theta \in [0,1]^n$ and $X\in \cal X^n$ we have
\begin{equation} \label{e:GammaXbd}
  \Gamma\pb{H+(d-1)^{-1/2}\,\theta \cdot X;z} \;\leq\; 2\Gamma(H;z) \,.
\end{equation}
To show \eqref{e:GammaXbd}, we use the resolvent identity to obtain (omitting the argument $z$ for brevity)
\begin{align*}
  &\absb{G_{ij}\pb{H+(d-1)^{-1/2} \theta \cdot X}}
  \\
  &\;=\; \absB{G_{ij}(H) - (d-1)^{-1/2} \pB{G(H)(\theta \cdot X)G\pb{H+(d-1)^{-1/2}\, \theta \cdot X}}_{ij}}\\
  &\;\leq\; \Gamma(H)+ 8n(d-1)^{-1/2} \Gamma(H)\Gamma\pb{H+(d-1)^{-1/2}\, \theta \cdot X}\\
  &\;\leq\; \Gamma(H)+\Gamma\pb{H+(d-1)^{-1/2}\, \theta \cdot X}/2\,.
\end{align*}
Taking the maximum over $i$ and $j$ yields \eqref{e:GammaXbd}.
Finally, \eqref{e:vXXbd} follows from \eqref{e:GammaXXbd}, as in the proof of \cite[Corollary~1.2]{1503.08702}.
\end{proof}

Note that since $G_{ij}(H;\bar z) = \overline{G_{ij}(H;z)}$, the estimates \eqref{e:Gbdt} and \eqref{e:GammaXXbd} for $\Gamma$ also hold
with $\eta<0$ if $\eta$ is replaced by $|\eta|$ on the right-hand sides. We shall use this tacitly in the following.

\subsection{Proof of Proposition~\ref{prop:comp}: eigenvalue correlation functions}

We now prove that the locally averaged local correlation functions of the matrix $H(0)|_{\f e^\perp}$
converge to those of $H(t)|_{\f e^\perp}$ for times $t \leq N^{-1-\delta}D^{1/2}$.
The main ingredient of the proof is the following lemma
comparing functions of Green's functions with spectral parameter $\eta$ slightly smaller than $1/N$.
Its proof follows easily from Proposition~\ref{prop:EF0Ft} and Lemma~\ref{prop:Gbdt}.
For random matrices with independent entries, analogous results were previously proved by the Green's function 
comparison theorem \cite{MR2981427}, and by direct analysis of the evolution of the matrix entries
under Dyson Brownian motion \cite{BY2016}.
We also remark that, in \cite{MR2784665}, eigenvalues are compared directly without involving
the Green's function.

\begin{lemma} \label{lem:Gcorr}
Fix $n \in \N$, and let $\kappa, \gamma,\delta > 0$ be sufficiently small.
Then the following holds for any $\eta \in [N^{-1-\gamma}, N^{-1}]$,
any sequence of positive integers $k_1,k_2,\dots,k_n$, any set of complex parameters
$z_j^m = E_j^m\pm \ii\eta$, where $j \in \qq{1,k_m}$, $m \in \qq{1,n}$, $|E_j^m| \leq  2-\kappa$,
and the $\pm$ signs are arbitrary.
Let $\phi \in C^\infty(\R^n)$ be a test function such that for any multi-index $m=(m_1,\cdots, m_n)$
with $1\leq |m|\leq 4$ and for any $\omega>0$ sufficiently small,
\begin{align}
\label{e:reg1}
\max\left\{|\partial^m \phi(x)| \col |x| \leq N^{\omega}\right\}
&\;\leq\; N^{O(\omega)} \,,
\\
\label{e:reg2}
\max\left\{|\partial^m \phi(x)| \col |x|\leq N^{2}\right\}
&\;\leq\; N^{O(1)} \,.
\end{align}
Then, with the notations $G_1(z) \deq G(H(0);z)$ and $G_2(z) \deq G(H(t);z)$, for any $t \leq D^{1/2} N^{-1-\delta}$, we have
\begin{multline} \label{e:Gcorr}
  \absa{  \E \phi\pa{N^{-k_1} \tr\pa{\prod_{j=1}^{k_1}G_1(z_j^1)}, \dots,
      N^{-k_n} \tr\pa{\prod_{j=1}^{k_n}G_1(z_j^n)} } -\E \phi(G_1 \rightarrow G_2) }
  \\
  \;=\;
  O(N^{-\delta/2+O(\gamma)})
  \,.
\end{multline}
Here, $\phi(G_1 \rightarrow G_2)$ is the expression obtained from the one to its left by replacing $G_1$ with $G_2$.
The implicit constants depend on $n$, $k_1, \dots, k_n$, $m_1, \dots, m_n$, and the constants in \eqref{e:reg1}--\eqref{e:reg2}.
\end{lemma}

\begin{proof}
For simplicity of notation, we show \eqref{e:Gcorr} only for $n=1$ and $k_1=1$;
the general case is analogous. 
We then write $z$ instead of $z_1^1$. To show the claim, it then suffices to show that
\begin{align}\label{e:gcompare}
\absa{  \E \phi \left(N^{-1} \tr G(H(t);z)\right) -  \E \phi \left(N^{-1} \tr G(H(0);z)\right)  }
\;=\; O\pb{tD^{-1/2}N^{1+\delta/2} N^{O(\gamma)}}
\,.
\end{align}
Set $F(H)\deq \phi\p{N^{-1}\tr G(H;z)}$.
We claim that if $r$ and $n$ are fixed (arbitrarily, independently of $N$), and if $t \leq D^{-1/4}$,
for any sufficiently large $N$ (depending on $r,n,\delta$), we have
\begin{equation} \label{e:derbound}
  \sup_{0 \leq s \leq t} \norm{\partial^nF}_{r,s}
  \;\leq\; N^{\delta/4+O(\gamma)}\,.
\end{equation}
Given \eqref{e:derbound}, Proposition~\ref{prop:EF0Ft} with $\varepsilon = \delta/4$ yields \eqref{e:gcompare}.

Thus, it only remains to show \eqref{e:derbound}.
Recall that the derivative of the Green's function in the direction of a matrix $X \in \cal X$ is given by $\partial_X G= -GXG$
(using that elements in $\cal X$ act on $\f e^\perp$).
Therefore, by the Leibniz rule,
for any $X_1, \dots, X_n \in \cal X$ and any $H \in \cal M$, we have
\begin{equation*}
  \partial_{X_1} \cdots \partial_{X_n}G \;=\; (-1)^n\sum_{\sigma\in S_n} GX_{\sigma(1)}G \cdots G X_{\sigma(n)} G \,,
\end{equation*}
where $S_n$ is the set of permutations of $n$ elements,
and we omit the dependence on $H$ on both sides in our notation.
Since (with respect to the standard basis of $\R^N$) 
each $X \in \cal X$ has at most $8$ nonvanishing entries,
and since these are in $\h{\pm 1}$, by definition of $\Gamma$ it follows that
\begin{equation*}
 \abs{N^{-1}\tr\partial_{X_1} \cdots \partial_{X_n}G}
  \;\leq\; N^{-1}\sum_{i=1}^{N} n! \max_{\sigma\in S_n}\left|\left( GX_{\sigma(1)}G \cdots G X_{\sigma(n)} G \right)_{ii}\right|
  \;\leq\; 8^nn! \Gamma^{n+1} \,.
\end{equation*}
From this and the chain rule, we obtain that there exist constants $C_n$ such that
\begin{equation} \label{e:phiderbd}
  |\partial_{X_1} \cdots \partial_{X_n}\phi(N^{-1}\tr G)| \;\leq\; C_n\Gamma^{2n} \max_{0\leq m\leq n} |\phi^{(m)}|  \,.
\end{equation}
By Corollary~\ref{cor:pertG} and since $|\eta| \geq N^{-1-\gamma}$, we have
$\sup_{\theta \in [0,1]^n}\sup_{X \in \cal X^n} \Gamma(H(s)+(d-1)^{-1/2}\theta\cdot X) \prec N^{\gamma}$, for any $0\leq s\leq t$.
For $n \leq 4$, by assumption \eqref{e:reg1} and \eqref{e:phiderbd} therefore
\begin{equation}\label{e:gset}
  \sup_{\theta \in [0,1]^n} \sup_{X \in \cal X^n} \absb{\partial_{X_1} \cdots \partial_{X_n}\phi\pb{N^{-1}\tr G\pb{H(s)+(d-1)^{-1/2}\theta\cdot X}}}
  \;\prec\; N^{O(\gamma)} \,.
\end{equation}
On the complement of the high-probability event of $\prec$ in \eqref{e:gset},
we use the trivial bound $\Gamma\leq \eta^{-1} \leq N^{1+\gamma}$ and \eqref{e:reg2}. We obtain
\begin{equation}\label{e:badset}
\sup_{\theta \in [0,1]^n} \sup_{X \in \cal X^n} \absb{\partial_{X_1}\cdots \partial_{X_n}\phi\pb{N^{-1}\tr G\pb{H(s)+(d-1)^{-1/2}\theta\cdot X}}}
\;\leq\; C_n \eta^{-2n}N^{O(1)} \;\leq\; N^{O(1)} \,,
\end{equation}
for any $0\leq s\leq t$. By combining the estimates \eqref{e:gset}--\eqref{e:badset}, for any constant $r=O(1)$, we have
\begin{equation} \label{e:DnFbdzeta}
\|\partial^nF\|_{r,s}
\;\leq\;
N^{1/\zeta+O(\gamma)} + N^{-\zeta/r+O(1)}
\;\leq\;
N^{\delta/4+O(\gamma)} \,,
\end{equation}
where $\zeta$ is as in Definition~\ref{def:highprob} and chosen sufficiently large,
depending on $r$.
This concludes the proof.
\end{proof}

The following lemma is essentially \cite[Theorem~6.4]{MR2981427}.
It transforms the statement about the Green's function of Lemma~\ref{lem:Gcorr}
to a statement about the local correlation functions. 

\begin{lemma}\label{lem:comparecorr}
Consider two random matrix ensembles $H_1$ and $H_2$ with Green's functions $G_1(z)$ and $G_2(z)$.
Suppose that, 
for all $\phi$ and parameters as in the statement of Lemma~\ref{lem:Gcorr}, 
the estimate \eqref{e:Gcorr} holds.
Then the local bulk eigenvalue correlation functions of $H_1$ and $H_2$ coincide.
\end{lemma}

\begin{proof}[Proof of Proposition~\ref{prop:comp}: correlation functions]
The proof follows directly by combining Lemmas~\ref{lem:Gcorr}--\ref{lem:comparecorr},
with $\delta$ given as in the assumption of Proposition \ref{prop:comp}.
\end{proof}

\subsection{Proof of Proposition~\ref{prop:comp}: eigenvalue gap statistics}

To prove that the eigenvalue gap statistics are stable for short times, we require
a weak level repulsion estimate. Such an estimate was derived in \cite[Theorem~4.1]{MR3429490}
for sparse matrices with independent entries, using a level repulsion estimate for
$t \geq N^{-1+c}$ established in \cite{1504.03605}. Here we adapt the proof of \cite[Theorem~4.1]{MR3429490}
to random regular graphs. The nontrivial dependence is dealt with by Proposition~\ref{prop:EF0Ft}.

If $\lambda_i(H)$ is a simple eigenvalue of $H|_{\f e^\perp}$,
we define
\begin{align}\label{e:Q}
Q_i(H)
\;\deq\; \frac{1}{N^2}\sum_{j:j\neq i,j\leq M}\frac{1}{(\lambda_j(H)-\lambda_i(H))^2}
\,,
\end{align}
and extend this definition by $Q_i(H)\deq\infty$ if $\lambda_i(H)$ is not a simple eigenvalue.
This quantity plays an important role in \cite{MR2784665}, where it is observed that
it captures the singularities of the derivatives of $\lambda_i(H)$.
In \cite{MR3429490}, it is found that $Q_i$ is stable under DBM and can thus be used to
show weak level repulsion from such an estimate for larger times (when a Gaussian component is present).

\begin{proposition}[Level repulsion]\label{prop:lvrp}
  Fix $\kappa>0$. Then for any sufficiently small $\tau>0$, 
  any $i \in \qq{\kappa N, (1-\kappa)N}$, and any $s\geq 0$, we have
  \begin{equation} \label{e:lvrpQ}
    \P\pb{Q_i(H(s)) \geq N^{2\tau}} \;=\; O(N^{-\tau/2}) \,.
  \end{equation}
  In particular,
  \begin{equation}\label{e:lvrp}
    \P\pb{\lambda_i(H(s)) - \lambda_{i+1}(H(s)) \leq N^{-1-\tau}} \;=\; O(N^{-\tau/2}) \,.
  \end{equation}
\end{proposition}

\begin{proof}
The proof is analogous to that of \cite[Theorem~4.1]{MR3429490}, with $H|_{\f e^\perp}$ instead of $H$.
We here focus on the differences.
These result from the replacement of \cite[Lemma~4.3]{MR3429490} by Proposition~\ref{prop:EF0Ft},
which takes into account the nontrivial correlation structure of the random regular graph.
As in \cite{MR3429490}, if $\lambda_i(H)$ is a simple eigenvalue of $H|_{\f e^\perp}$,
we define the matrix
\begin{equation*}
R_i(H) \;\deq\; \sum_{j: j\neq i,  j\leq M} \frac{1}{\lambda_i(H)-\lambda_j(H)}\f v_j(H)\f v_j(H)^* 
\;=\;\frac{1}{2\pi i} \oint_{|z-\lambda_i(H)|=\omega} \frac{G(H;z)}{\lambda_i(H)-z} \, \dd z\,,
\end{equation*}
where $\omega$ is chosen such that the contour $|z-\lambda_i(H)|=\omega$ encloses only $\lambda_i(H)$.
Then we have
\begin{equation*}\label{e:QR}
Q_i(H)
\;=\; \frac{1}{N^2}\tr(R_i(H)^2)
\,.
\end{equation*}

Given $\tau>0$, define a cutoff function $\chi$
satisfying the following two properties:
(1) $\chi$ is smooth, and the first four derivatives are bounded, i.e.\ $|\chi^{(k)}(x)|=O(1)$, for $k=1,2,3,4$;
(2) On the interval $[0,N^{2\tau}]$, $|\chi(x)-x|\leq 1$, and for $x\geq N^{2\tau}$, $\chi(x)=N^{2\tau}$.
Then $\chi \circ Q_i$ extends to a smooth function on the space of symmetric matrices.

The proof of \eqref{e:lvrpQ} consists of three steps. The first step is the estimate 
\begin{align}\label{e:init}
\E[\chi(Q_i(H(s)))] \;=\; O(N^{3\tau/2}) \,,
\end{align}
for $s\geq t \deq N^{-1+c}$.
This estimate follows from \cite[Theorem~3.6]{1504.03605},
whose assumptions are satisfied with high probability for the random $d$-regular graph by Proposition~\ref{prop:Gbdt}.
In particular, independence of the entries of $H$ is not used.

In the second step, we derive the comparison estimate
\begin{align}\label{e:continuity}
\absb{\E[\chi(Q_i(H(t)))]-\E[\chi(Q_i(H(s)))]} \;\leq\; 1 \,,
\end{align}
for $s \in [0,t]$.
Instead of using \cite[Lemma~4.3]{MR3429490}, which requires that the entries of the random matrix $H(s)$ are independent,
we use Proposition \ref{prop:EF0Ft}, which takes into account the nontrivial correlation structure of the random regular graph.
By Proposition \ref{prop:EF0Ft} with $F(H)\deq \chi(Q_i(H))$,
it suffices to bound
\begin{align}\label{e:derb}
\|\partial^n F\|_{r,s} \;=\;
\E\left[\sup_{\theta\in[0,1]^n} \sup_{X \in \cal X^n }\left|\partial_{X_1}\partial_{X_2}\cdots \partial_{X_n} F\pb{H(s)+(d-1)^{-1/2}\, \theta \cdot X}\right|^r\right]^{1/r},
\end{align}
for any (large) fixed integer $r$ and $n=1,2,3,4$.
To this end, the computation of the proof of \cite[Proposition~4.6]{MR3429490} applies,
by simply replacing the derivatives $\partial_{ab}^{(n)}$ by $\partial_{X_1} \cdots \partial_{X_n}$ with $X_l \in \cal X$.
Here the formulas \cite[(4.16)--(4.18)]{MR3429490} remain valid after replacing $V$ by the $X_l$ appropriately,
and similarly the formula below \cite[(4.18)]{MR3429490} remains valid after replacing $V_{ij}$ by $\f v_i^*(H) X_l \f v_j(H)$.
Moreover, an analogous formula holds for $n=4$; see e.g.\ \cite[p.8]{MR0493421}.
The same formulas are valid with $H$ replaced by $H + (d-1)^{-1/2}\, \theta \cdot X$.
Since the $X_l$ have only $8$ nonvanishing entries (in the standard basis on $\R^N$), and these are equal to $\pm 1$,
Corollary~\ref{cor:pertG} then implies
\begin{equation*}
\sup_{\theta \in [0,1]^n}\sup_{X \in \cal X^n} \absb{\f v_i^*\pb{H(s) + (d-1)^{-1/2}\, \theta \cdot X} \, X_l  \, \f v_j\pb{H(s)+(d-1)^{-1/2}\, \theta \cdot X}} \;\prec\; N^{-1}
\end{equation*}
for any $s \in [0,t]$.
As in the proof of \cite[Proposition~4.6]{MR3429490}, we therefore get
\begin{equation*}
\sup_{\theta\in[0,1]^n} \sup_{X \in \cal X^n }\absb{ \partial_{X_1}\partial_{X_2}\cdots \partial_{X_n} F\pb{H(s)+(d-1)^{-1/2}\, \theta \cdot X} }
\;\prec\; N^{(n+2)\tau} \,.
\end{equation*}
From this, bounding \eqref{e:derb} as in \eqref{e:DnFbdzeta}, we obtain
\begin{equation} \label{e:derbres}
  \|\partial^n F\|_{r,s} \;\leq\; N^{c+(n+2)\tau} \,.
\end{equation}
for arbitrarily small $c>0$ and $N$ large enough.
Then \eqref{e:continuity} follows from Proposition~\ref{e:EF0Ft} since
$O(tD^{-1/2}N)N^{c+6\tau} \leq O(N^{-\alpha/2+2c+6\tau})\leq 1$ for $t \leq N^{-1+c}$
and $D \geq N^{\alpha}$,
by chooosing $c$ and $\tau$ sufficiently small. 

In the last step, we combine \eqref{e:init} and \eqref{e:continuity}, and thus obtain
\begin{align*}
\E[\chi(Q_i(H(s)))] \;=\; O(N^{3\tau/2}) \,,
\end{align*}
for any $s \geq 0$.
Then \eqref{e:lvrpQ} follows easily by Markov's inequality and the definition of $\chi$.
\end{proof}

\begin{proof}[Proof of Proposition~\ref{prop:comp}: gap statistics]
Throughout the proof, we use the abbreviation $\lambda_i(t) \equiv \lambda_i(H(t))$.
Fix $\kappa>0$, $\delta>0$, and $t \leq N^{-1-\delta}D^{1/2}$.
Since $\varrho(\gamma_i)$ is bounded above and below for $i \in \qq{\kappa N, (1-\kappa)N}$,
it suffices to prove \eqref{e:gap2} with $\varrho(\gamma_{i})$ replaced by $1$.
Moreover, for any $n \in \N$ and $\phi \in C^\infty(\R^n)$ with bounded first four derivatives, it suffices to show
the stronger claim
\begin{equation} \label{e:phiconv}
\E \phi\pb{N\lambda_i(0), \dots, N\lambda_{i+n}(0)}
\;=\;
\E \phi\pb{N\lambda_i(t), \dots, N\lambda_{i+n}(t)}
+
o(1)
\end{equation}
as $N \to \infty$, uniformly in $i \in \qq{\kappa N, (1-\kappa)N}$.
For simplicity of notation, we only prove \eqref{e:phiconv} for $n=1$; the general case is analogous
and we comment on the differences at the end of the proof.
Thus, for any $i\in \qq{\kappa N, (1-\kappa)N}$ and $\phi \in C^\infty(\R)$ with bounded first four derivatives, we show
\begin{equation}\label{e:oneg}
\E[\phi(N\lambda_{i}(0))]-\E[\phi(N\lambda_{i}(t))] \;=\; o(1) \,.
\end{equation} 

Given a small constant $\tau>0$,
we choose a cutoff function $\rho$ such that $\rho(x)=1$ for $x\leq N^{2\tau}$ and $\rho(x)=0$ for $x\geq 2N^{2\tau}$.
Using \eqref{e:lvrpQ}, we can first remove a bad event on which $Q_i$ is large:
\begin{align*}
&\left|\E[\phi(N\lambda_{i}(0))]-\E[\phi(N\lambda_{i}(t))]\right|\\
&\;\leq\; \absb{\E[\phi(N\lambda_{i}(0))\rho(Q_i(H(0)))]-\E[\phi(N\lambda_{i}(t))\rho(Q_i(H(t)))]}\\
&\qquad +\|\phi\|_{\infty}\pb{\P(Q_i(H(0))\geq N^{2\tau})+\P(Q_i(H(t))\geq N^{2\tau})}\\
&\;\leq\; \absb{\E[\phi(N\lambda_{i}(0))\rho(Q_i(H(0)))]-\E[\phi(N\lambda_{i}(t))\rho(Q_i(H(t)))]}+O\pbb{\frac{\|\phi\|_{\infty}}{N^{\tau/2}}} \,.
\end{align*}
To estimate the right-hand side, 
we apply Proposition~\ref{prop:EF0Ft} with $F(H)\deq \phi(N\lambda_{i}(H))\rho(Q_i(H))$.
By an argument analogous to that used to obtain \eqref{e:derbres},
for any $r$ and $n=1,2,3,4$, we find the bound
\begin{equation}\label{e:derb2}
\|\partial^n F\|_{r,s} 
\;\leq\; N^{c+O(\tau)}
\end{equation}
for arbitrarily small $c>0$ (and $N$ sufficiently large).
More precisely, by the product rule, the derivatives act either on $\phi(N\lambda_i)$ or $\rho \circ Q_i$.
In the bound of any of these derivatives, by definition of $\rho$, we can assume that $Q_i \leq 2N^{2\tau}$.
Then the derivatives of $\rho \circ Q_i$ are bounded exactly as in the proof of Proposition~\ref{prop:lvrp}.
For the derivatives of $\phi(N\lambda_i)$, by the chain rule and since $\phi$ is smooth, it suffices to 
bound the derivatives of the eigenvalues $\lambda_i$.
This is again done similarly to the bounds on the derivatives of $Q_i$.
Indeed, the derivatives of the eigenvalues can be expressed in terms of the eigenvalues and eigenvectors
as done in \cite[(4.16)--(4.18)]{MR3429490} (and with \cite[p.8]{MR0493421} for $n=4$).
The latter expressions are bounded using the delocalization of
eigenvectors \eqref{e:Gbdt}, and using that
\begin{equation*}
  \sum_{j: j \neq i} \frac{1}{|\lambda_i(s)-\lambda_j(s)|} \;\prec\; N Q_i^{1/2}(H(s))\,,\qquad
  \sum_{j: j \neq i} \frac{1}{|\lambda_i(s)-\lambda_j(s)|^k} \;\leq\; N^k Q_i^{k/2}(H(s))\,,
\end{equation*}
as in \cite[(4.11)--(4.12)]{MR3429490}.

As a consequence of Proposition~\ref{prop:EF0Ft} and \eqref{e:derb2},
with $t \leq N^{-1-\delta}D^{1/2}$, we finally obtain
\begin{equation*}
  \absa{ \E[\phi(N\lambda_{i}(0))\rho(Q_i(H(0)))]-\E[\phi(N\lambda_{i}(t))\rho(Q_i(H(t)))]  }
  \;=\; O(N^{c+O(\tau)-\delta}) \,,
\end{equation*}
and \eqref{e:oneg} then follows by taking $c$ and $\tau$ small enough that $c + O(\tau) < \delta$.

In the general case of a test function $\phi(N\lambda_i, \dots, N\lambda_{i+n})$,
we use the product of cutoff functions
$(\rho \circ Q_i) \cdots (\rho \circ Q_{i+n})$ instead of $\rho \circ Q_i$, and proceed
otherwise analogously.
\end{proof}

\subsection{Proof of Proposition~\ref{prop:univ}}
\label{sec:univ}

\begin{proof}[Proof of Propositions~\ref{prop:univ}]
Given the estimates \eqref{e:Gbdt}--\eqref{e:rig}, the same argument as in \cite[Section~3]{MR3429490} applies.
\end{proof}

\section*{Acknowledgements}

AK was partly supported by Swiss National Science Foundation grant 144662.
HTY was partly supported by NSF grant DMS-1307444 and a Simons Investigator fellowship.
We thank Ben Landon for helpful discussions,
Peter Sarnak for informing us about \cite{MR1691538,MR2433888},
and Michael Aizenman for bringing the works \cite{1501.04907,1503.06417} to our attention.

\bibliography{all}
\bibliographystyle{plain}

\end{document}